\documentclass[reqno]{amsart}
\usepackage{amssymb, latexsym}
\usepackage{hyperref}

\allowdisplaybreaks

\newcommand{\Z}{{\mathbb{Z}}}
\newcommand{\C}{{\mathbb{C}}}
\newcommand{\R}{{\mathbb{R}}}

\newcommand{\ird}{ I\times\R^d}

\newcommand{\eps}{{\varepsilon}}

\newcommand{\gnj}{g_n^j}
\newcommand{\tnj}{t_n^j}

\newcommand{\wnJ}{w_n^J}

\theoremstyle{plain}
\newtheorem{theorem}{Theorem}
\newtheorem{proposition}[theorem]{Proposition}
\newtheorem{lemma}[theorem]{Lemma}
\newtheorem{corollary}[theorem]{Corollary}
\newtheorem{conjecture}[theorem]{Conjecture}
\theoremstyle{definition}
\newtheorem{definition}[theorem]{Definition}
\newtheorem{remark}[theorem]{Remark}
\newtheorem*{remarks}{Remarks}

\newcounter{smalllist}

\newenvironment{CI}{\begin{list}{{\ $\bullet$\ }}{%
\setlength{\topsep}{0mm}\setlength{\parsep}{0mm}\setlength{\itemsep}{0mm}%
\setlength{\labelwidth}{0mm}\setlength{\itemindent}{0mm}\setlength{\leftmargin}{0mm}%
\setlength{\labelsep}{0mm}}}{\end{list}}

\numberwithin{equation}{section} \numberwithin{theorem}{section}
\begin{document}

\title[Energy-supercritical NLS: critical $\dot H^s$-bounds imply scattering]{Energy-supercritical NLS: \\critical $\dot H^s$-bounds imply scattering}
\author{Rowan Killip}
\address{University of California, Los Angeles, CA}
\author{Monica Visan}
\address{University of Chicago, Chicago, IL}

\begin{abstract}
We consider two classes of defocusing energy-supercritical nonlinear Schr\"odinger equations in dimensions $d\geq 5$.
We prove that if the solution $u$ is apriorily bounded in the critical Sobolev space, that is, $u\in L_t^\infty \dot H^{s_c}_x$,
then $u$ is global and scatters.
\end{abstract}

\maketitle

\section{Introduction}
We consider the initial-value problem for the defocusing nonlinear Schr\"odinger equation in dimension $d\geq 5$,
\begin{equation}\label{nls}
\begin{cases}
\ iu_t=-\Delta u+F(u)\\
\ u(t=0,x)=u_0(x),
\end{cases}
\end{equation}
where the nonlinearity $F(u)= |u|^pu$ is energy-supercritical, that is, $p>\tfrac4{d-2}$.

The class of solutions to \eqref{nls} is left invariant by the scaling
\begin{equation}\label{scaling}
u(t,x)\mapsto \lambda^\frac2{p} u(\lambda^2 t,\lambda x).
\end{equation}
This defines a notion of \emph{criticality}.  More precisely, a quick computation shows that the only homogeneous $L_x^2$-based Sobolev space
left invariant by the scaling is $\dot H^{s_c}_x(\R^d)$, where the \emph{critical regularity} is $s_c:=\frac d2-\frac 2p$.
If the regularity of the initial data to \eqref{nls} is higher/lower than the critical regularity $s_c$, we call the problem
\emph{subcritical/supercritical}.

We consider \eqref{nls} for initial data belonging to the critical homogeneous Sobolev space, that is, $u_0\in\dot H^{s_c}_x(\R^d)$, in two regimes
where $s_c>1$ and $d\geq 5$.  We prove that any maximal-lifespan solution that is uniformly bounded (throughout its lifespan) in $\dot H^{s_c}_x(\R^d)$
must be global and scatter.  We were prompted to consider this problem by a recent preprint of Kenig and Merle \cite{kenig-merle:wave sup} which
proves similar results for radial solutions to the nonlinear wave equation in $\R^3$.

Let us start by making the notion of a solution more precise.

\begin{definition}[Solution]\label{D:solution} A function $u: I \times \R^d \to \C$ on a non-empty time interval $0\in I \subset \R$
is a \emph{solution} (more precisely, a strong $\dot H^{s_c}_x(\R^d)$ solution) to \eqref{nls} if it lies in the class $C^0_t \dot
H^{s_c}_x(K \times \R^d) \cap L^{p(d+2)/2}_{t,x}(K \times \R^d)$ for all compact $K \subset I$, and obeys the Duhamel formula
\begin{align}\label{old duhamel}
u(t) = e^{it\Delta} u(0) - i \int_{0}^{t} e^{i(t-s)\Delta} F(u(s))\, ds
\end{align}
for all $t \in I$.  We refer to the interval $I$ as the \emph{lifespan} of $u$. We say that $u$ is a \emph{maximal-lifespan
solution} if the solution cannot be extended to any strictly larger interval. We say that $u$ is a \emph{global solution} if $I= \R$.
\end{definition}

We define the \emph{scattering size} of a solution to \eqref{nls} on a time interval $I$ by
$$
S_I(u):= \int_I \int_{\R^d} |u(t,x)|^{\frac{p(d+2)}2}\, dx \,dt.
$$

Associated to the notion of solution is a corresponding notion of blowup.  As we will see in Theorem~\ref{T:local} below, this
precisely corresponds to the impossibility of continuing the solution.

\begin{definition}[Blowup]\label{D:blowup}
We say that a solution $u$ to \eqref{nls} \emph{blows up forward in time} if there exists a time $t_1 \in I$ such that
$$ S_{[t_1, \sup I)}(u) = \infty$$
and that $u$ \emph{blows up backward in time} if there exists a time $t_1 \in I$ such that
$$ S_{(\inf I, t_1]}(u) = \infty.$$
\end{definition}

We subscribe to the following conjecture.

\begin{conjecture}\label{C:conj}
Let $d\geq 1$, $p\geq \frac 4d$, and let $u:I\times\R^d\to\C$ be a maximal-lifespan solution to \eqref{nls} such that
$u\in L_t^\infty \dot H^{s_c}_x$.  Then $u$ is global and moreover,
\begin{equation}\label{scattering finite}
S_\R(u)\leq C\bigl(\|u\|_{L_t^\infty \dot H^{s_c}_x}\bigr),
\end{equation}
for some function $C:[0,\infty)\to [0,\infty)$.
\end{conjecture}

Our primary goal in this paper is to demonstrate how techniques developed to treat the energy-critical NLS can be applied to
Conjecture~\ref{C:conj} in the regime $s_c\geq 1$, although some of the arguments we will use were developed first in the mass-critical setting.
As we will describe, the appearance of the $L_t^\infty \dot H^{s_c}_x$ norm on the right-hand side of \eqref{scattering finite} renders
illusory the supercriticality of the equation.  The famed supercriticality of Navier--Stokes is the fact that the problem is supercritical with respect
to all quantities controlled by (known) conservation/monotonicity laws; see, for instance, the discussion in \cite{Tao:NS}.  In the context
of Conjecture~\ref{C:conj}, the assumption that the solution is uniformly bounded in $\dot H^{s_c}_x$ plays the role of the missing
critical conservation law.  It is not surprising therefore that techniques developed to treat problems with true critical conservation laws
should be applicable in this setting.  Next we review some of this work before describing the particular contribution of this paper.

Mass and energy are the only known coercive conserved quantities for NLS; hence, the corresponding critical NLS equations have received the most
attention. In the mass-critical case, the critical regularity is $s_c=0$ (i.e. $p=\frac 4d$) and the scaling \eqref{scaling} leaves the mass invariant,
that is, the conserved quantity
$$
M(u):=\int_{\R^d} |u(t,x)|^2\, dx.
$$
Similarly, in the energy-critical case, the critical regularity is $s_c=1$ (i.e. $p=\frac4{d-2}$, $d\geq 3$) and the scaling \eqref{scaling} leaves
invariant the energy,
$$
E(u):=\int_{\R^d} \bigl[|\nabla u(t,x)|^2 + \tfrac1{2+p}|u(t,x)|^{2+p} \bigr]\, dx,
$$
which is also a conserved quantity for \eqref{nls}.

In the defocusing energy-critical case, it is known that all $\dot H^1_x$ initial data lead to global solutions with finite scattering size.
Indeed, this was proved by Bourgain \cite{borg:scatter}, Grillakis \cite{grillakis}, and Tao \cite{tao: gwp radial} for spherically-symmetric
initial data, and by Colliander--Keel--Staffilani--Takaoka--Tao \cite{ckstt:gwp}, Ryckman--Visan \cite{RV}, and Visan \cite{thesis:art, Monica:thesis}
for arbitrary initial data.  For results in the focusing case see \cite{kenig-merle, Berbec}.

Unlike for the nonlinear wave equation (NLW), all known monotonicity formulae for NLS (that is, Morawetz-type inequalities) scale differently than the
energy.  Ultimately, the ingenious induction on energy technique of Bourgain (and the concomitant identification of bubbles) introduces a length
scale to the problem which makes it possible to use non-invariantly scaling monotonicity formulae.  All subsequent work has built upon this
foundational insight.

The Lin--Strauss Morawetz inequality used by Bourgain has $\dot H^{1/2}_x$-scaling and is best adapted to the spherically-symmetric problem.
To treat the non-radial problem, Colliander, Keel, Staffilani, Takaoka, and Tao introduced an interaction Morawetz inequality; this has
$\dot H^{1/4}_x$-scaling, which is even further from the $\dot H^1_x$-scaling of the energy-critical problem.  As the reader may notice, in both cases
the regularity associated to the monotonicity formulae is lower than the critical regularity of the equation.  This is the philosophical
basis of our belief that the techniques developed for treating the energy-critical problem should be broadly applicable to Conjecture~\ref{C:conj}
whenever $s_c\geq 1/2$.  In this paper, we will by no means complete this program, but rather have chosen to present some selected results
that give the flavor of our main thesis without becoming swamped with technicalities.

We turn our attention now to the defocusing mass-critical NLS.  In this case, Conjecture~\ref{C:conj} has been proved for spherically-symmetric
$L_x^2$ initial data in all dimensions $d\geq 2$; see \cite{KTV, KVZ, TVZ:sloth}.  For a proof of the corresponding conjecture in the focusing case
(for spherically-symmetric $L_x^2$ data with mass less than that of the ground state and $d\geq 2$) see \cite{KTV, KVZ}.  At present, we do not
know how to deal with the Galilean symmetry possessed by this equation, except through suppressing it by assuming spherical symmetry.
We also note that in this case, one needs to prove additional regularity (rather than decay) to gain access to the known monotonicity formulae.

The first instance of Conjecture~\ref{C:conj} to fall at non-conserved critical regularity was the case $s_c=1/2$ in dimension $d=3$.
This was achieved by Kenig and Merle, \cite{kenig-merle:1/2}.  They used the concentration-compactness technique in the manner they pioneered in
\cite{kenig-merle} together with the Lin--Strauss Morawetz inequality.

The present paper is motivated by a recent preprint of Kenig and Merle, \cite{kenig-merle:wave sup}, who consider spherically-symmetric solutions
to a class of defocusing energy-supercritical nonlinear wave equations in three dimensions.  They prove that if the solution is known to be
uniformly bounded in the critical $\dot H^s_x$-space throughout its lifetime, then the solution must be global and it must scatter;
this is precisely the NLW analogue of Conjecture~\ref{C:conj}.

Earlier, we drew a parallel to the Navier--Stokes equation.  The most natural analogue of Conjecture~\ref{C:conj} in that setting is to show
that boundedness of a critical norm implies global regularity.  Such results are known; see \cite{Escauriaza} and the references therein.

In this paper we prove Conjecture~\ref{C:conj} in several instances of defocusing energy-supercritical nonlinear
Schr\"odinger equations in dimensions $d\geq 5$ for arbitrary (not necessarily spherically-symmetric) initial data.
The approach we take is modeled after \cite{Berbec}, which considers the energy-critical problem in dimensions $d\geq 5$ in both
the defocusing and focusing cases.

We will consider two different settings.  In the first one, the nonlinearity is cubic, that is, $F(u)=|u|^2u$, and hence the
critical regularity is $s_c=\frac{d-2}2$.  Note that in dimension $d\geq 5$, this problem is energy-supercritical.

The second problem we consider is that of general (not necessarily polynomial) energy-supercritical nonlinearities in dimension $d\geq 5$,
that is, $p>\frac4{d-2}$.  In this case, we also impose some additional constraints on the power $p$.  First, we ask that the nonlinearity
obeys a certain smoothness condition; more precisely, we ask that $s_c<1+p$, which is equivalent to $2p^2-p(d-2)+4>0$.  The role of this
constraint is to allow us to take $s_c$-many derivatives of the nonlinearity $F(u)$; this is important in the development of the local theory.
Moreover, in Section~\ref{S:neg}, we require that $s_c$ and $p$ obey some further constraints.  Together, these amount to
\begin{equation}\label{sc}
\begin{cases}
\ 1< s_c< \tfrac32 \quad &\text{for } d=5,6\\
\ 1< s_c<\tfrac{d+2-\sqrt{(d-2)^2-16}}4 \quad &\text{for } d\geq 7.
\end{cases}
\end{equation}
One should not view \eqref{sc} as a major constraint on the size of the critical regularity $s_c$.  Indeed, we claim that an interpolation
of the techniques we present to treat the two problems outlined above can be used to treat any defocusing energy-supercritical NLS (with
the solution apriorily bounded in $\dot H^{s_c}_x$) in dimensions $d\geq 5$, without any additional constraint on $s_c$ if the power
$p$ is an even integer and requiring merely the smoothness condition $s_c<1+p$ for arbitrary powers $p$.  However, for the sake of readability, we chose
not to work in this greater generality.

Our main results are the following:

\begin{theorem}[Spacetime bounds -- the cubic]\label{T:main cubic}
Let $d\geq 5$ and $F(u)=|u|^2u$.  Let $u:I\times\R^d\to\C$ be a maximal-lifespan solution to \eqref{nls}
such that $u\in L_t^\infty \dot H^{\frac{d-2}2}_x(I\times\R^d)$.  Then $u$ is global and moreover,
$$
S_\R(u)\leq C\bigl(\|u\|_{L_t^\infty \dot H^{\frac{d-2}2}_x}\bigr).
$$
\end{theorem}

\begin{theorem}[Spacetime bounds]\label{T:main}
Let $d\geq 5$ and assume the critical regularity $s_c$ satisfies \eqref{sc}.  Let $u:I\times\R^d\to\C$ be a maximal-lifespan solution to \eqref{nls}
such that $u\in L_t^\infty \dot H^{s_c}_x(I\times\R^d)$.  Then $u$ is global and moreover,
$$
S_\R(u)\leq C\bigl(\|u\|_{L_t^\infty \dot H^{s_c}_x}\bigr).
$$
\end{theorem}

As we already mentioned, the proofs of Theorems~\ref{T:main cubic} and \ref{T:main} follow closely the approach taken in
\cite{Berbec} to study the energy-critical problem.  We outline the argument in subsection~\ref{SS:outline} below.  Our decision
to work in dimensions $d\geq 5$ was motivated by the fact that it allows us to employ some techniques used in \cite{Berbec}, in
particular, the double Duhamel trick.  The natural approach in lower dimensions would be to use the frequency localized
interaction Morawetz inequality in the spirit of \cite{ckstt:gwp}. While subsequent developments (some of which are reviewed in
this paper) lead to simplifications of \cite{ckstt:gwp}, this would still be a significant undertaking and we do not pursue it
here.

The arguments presented here apply mutis mutandis to the corresponding Hartree equations; indeed, the fact that the nonlinearity depends
polynomially on $u$ for that equation means that it resembles the simpler cubic case treated here.  We also believe that the arguments
adapt to the corresponding energy-supercritical wave equations in dimensions $d\geq 6$; however, we have not worked through the details.

To study the global theory for \eqref{nls}, we must first develop a local theory for this equation.  To this end, we revisit arguments
by Cazenave and Weissler, \cite{cwI}, who treated the case $0\leq s_c\leq 1$, as well as more sophisticated stability results in the
spirit of \cite{Nakanishi, TV}.

\begin{theorem}[Local well-posedness]\label{T:local}
Let $d$ and $s_c$ be as in Theorem~\ref{T:main cubic} or \ref{T:main}.  Then, given $u_0\in\dot
H_x^{s_c}(\R^d)$ and $t_0\in \R$, there exists a unique maximal-lifespan solution $u:\ird \rightarrow \C$ to \eqref{nls} with
initial data $u(t_0)=u_0$.  This solution also has the following properties:
\begin{CI}
\item (Local existence) $I$ is an open neighbourhood of $t_0$.
\item (Blowup criterion) If $\sup I$ is finite, then $u$ blows up forward in time (in the sense of Definition~\ref{D:blowup}).  If $\inf I$ is finite,
then $u$ blows up backward in time.
\item (Scattering) If $\sup I=+\infty$ and $u$ does not blow up forward in time, then $u$ scatters forward in time, that is,
there exists a unique $u_+ \in \dot H^{s_c}_x(\R^d)$ such that
\begin{equation}\label{like u+}
\lim_{t \to +\infty} \| u(t)-e^{it\Delta} u_+ \|_{\dot H^{s_c}_x(\R^d)} = 0.
\end{equation}
Conversely, given $u_+ \in \dot H^{s_c}_x(\R^d)$ there is a unique solution to \eqref{nls} in a neighbourhood of infinity so that
\eqref{like u+} holds.
\item (Small data global existence) If $\bigl\||\nabla|^{s_c} u_0\|_2$ is sufficiently small (depending on $d,p$), then $u$ is a global solution
which does not blow up either forward or backward in time.  Indeed, in this case $S_\R(u)\lesssim \bigl\| |\nabla|^{s_c} u_0\bigr\|_2^{p(d+2)/2}$.
\end{CI}
\end{theorem}

In Section~\ref{S:Local theory} we establish this theorem as a corollary of our stability results Theorems~\ref{T:stab cubic} and \ref{T:stab}.
These stability results are essential to the arguments we present, more specifically, to the proof of Theorem~\ref{T:reduct}.

\subsection{Outline of the proofs of Theorems~\ref{T:main cubic} and \ref{T:main}}\label{SS:outline}

We argue by contradiction.  We show that if either Theorem~\ref{T:main cubic} or Theorem~\ref{T:main} failed, this would imply the existence
of a very special type of counterexample.  Such counterexamples are then shown to have a wealth of properties not immediately
apparent from their construction, so many properties, in fact, that they cannot exist.

While we will make some further reductions later, the main property of the special counterexamples is almost periodicity modulo
symmetries:

\begin{definition}[Almost periodicity modulo symmetries]\label{D:ap}
Suppose $s_c>0$.  A solution $u$ to \eqref{nls} with lifespan $I$ is said to be \emph{almost periodic modulo symmetries} if there
exist functions $N: I \to \R^+$, $x:I\to \R^d$, and $C: \R^+ \to \R^+$ such that for all $t \in I$ and $\eta > 0$,
$$ \int_{|x-x(t)| \geq C(\eta)/N(t)} \bigl||\nabla|^{s_c} u(t,x)\bigr|^2\, dx
    +\int_{|\xi| \geq C(\eta) N(t)} |\xi|^{2s_c}\, | \hat u(t,\xi)|^2\, d\xi\leq \eta.
$$
We refer to the function $N$ as the \emph{frequency scale function} for the solution $u$, $x$ the \emph{spatial center
function}, and to $C$ as the \emph{compactness modulus function}.
\end{definition}

\begin{remark}
The parameter $N(t)$ measures the frequency scale of the solution at time $t$, while $1/N(t)$ measures the spatial scale.
It is possible to multiply $N(t)$ by any function of $t$ that is bounded both above and below, provided that we also modify
the compactness modulus function $C$ accordingly.
\end{remark}

\begin{remark}
When $s_c=0$ the equation admits a new symmetry, namely, Galilei invariance.  This introduces a \emph{frequency center function} $\xi(t)$
in the definition of almost periodicity modulo symmetries; see \cite{keraani-l2, tvz:cc} for further discussion.
\end{remark}

\begin{remark}\label{R:pot energy}
By the Ascoli--Arzela Theorem, a family of functions is precompact in $\dot H^{s_c}_x(\R^d)$ if and only if it is norm-bounded and
there exists a compactness modulus function $C$ so that
$$
\int_{|x| \geq C(\eta)} \bigl||\nabla|^{s_c} f(x)\bigr|^2\ dx + \int_{|\xi| \geq C(\eta)} |\xi|^{2s_c} \, |\hat f(\xi)|^2\ d\xi \leq \eta
$$
for all functions $f$ in the family.  Thus, an equivalent formulation of Definition~\ref{D:ap} is as follows: $u$ is almost
periodic modulo symmetries if and only if
$$
\{ u(t): t \in I \} \subseteq \{ \lambda^{\frac2p} f(\lambda (x+x_0)) : \, \lambda\in(0,\infty), \ x_0\in \R^d, \text{ and
}f \in K \}
$$
for some compact subset $K$ of $\dot H^{s_c}_x(\R^d)$.
\end{remark}

\begin{remark}\label{R:c small}
A further consequence of compactness modulo symmetries is the existence of a function $c: \R^+\to \R^+$ so that
$$
\int_{|x-x(t)| \leq c(\eta)/N(t)} \bigl||\nabla|^{s_c} u(t,x)\bigr|^2\, dx
    + \int_{|\xi| \leq c(\eta) N(t)} |\xi|^{2s_c}\, | \hat u(t,\xi)|^2\, d\xi \leq \eta
$$
for all $t \in I$ and $\eta > 0$.
\end{remark}

With these preliminaries out of the way, we can now describe the first major milestone in the proof of Theorems~\ref{T:main cubic} and \ref{T:main}.

\begin{theorem}[Reduction to almost periodic solutions]\label{T:reduct}\quad
Suppose that Theorem~\ref{T:main cubic} (or Theorem~\ref{T:main}) failed.  Then there exists a maximal-lifespan
solution $u:\ird\to \C$ to \eqref{nls} such that $u\in L_t^\infty \dot H^{s_c}_x(I\times\R^d)$, $u$ is almost periodic modulo symmetries,
and $u$ blows up both forward and backward in time. Moreover, $u$ has minimal $L_t^\infty \dot H^{s_c}_x$-norm among all blowup solutions,
that is,
$$
\sup_{t\in I} \bigl\| |\nabla|^{s_c} u(t)\bigr\|_2 \leq \sup_{t\in J}\bigl\| |\nabla|^{s_c} v(t)\bigr\|_2
$$
for all maximal-lifespan solutions $v:J\times\R^d \to \C$ that blow up in at least one time direction.
\end{theorem}

The reduction to almost periodic solutions was first realized by Keraani \cite{keraani-l2} in the context of the mass-critical NLS.  This was adapted
to the energy-critical case by Kenig and Merle \cite{kenig-merle}.

Theorem~\ref{T:reduct} provides no information on the modulation parameters $x(t)$ and $N(t)$.  In order to treat the mass-critical NLS
in two dimensions, a further reduction was introduced in \cite{KTV} where the behaviour of $N(t)$ is heavily constrained.  This argument
wa adapted to the energy-critical case in \cite{Berbec}.  This latter argument is directly applicable to the setting of this paper and yields

\begin{theorem}[Three special scenarios for blowup]\label{T:enemies}
Suppose that Theorem~\ref{T:main cubic} (or Theorem~\ref{T:main}) failed.  Then there exists a maximal-lifespan solution $u:I\times\R^d\to \C$,
which obeys $u\in L_t^\infty \dot H^{s_c}_x(I\times\R^d)$, is almost periodic modulo symmetries, and $S_I(u)=\infty$.
Moreover, we can also ensure that the lifespan $I$ and the frequency scale function $N:I\to\R^+$ match one of the following three scenarios:
\begin{itemize}
\item[I.] (Finite-time blowup) We have that either $\sup I<\infty $ or $|\inf I|<\infty$.
\item[II.] (Soliton-like solution) We have $I = \R$ and
\begin{equation*}
 N(t) = 1 \quad \text{for all} \quad t \in \R.
\end{equation*}
\item[III.] (Low-to-high frequency cascade) We have $I = \R$,
\begin{equation*}
\inf_{t \in \R} N(t) \geq 1, \quad \text{and} \quad \limsup_{t \to +\infty} N(t) = \infty.
\end{equation*}
\end{itemize}
\end{theorem}

Therefore, in order to prove Theorems~\ref{T:main cubic} and \ref{T:main} it suffices to preclude the existence of solutions that satisfy
the criteria in Theorem~\ref{T:enemies}.  Following \cite{Berbec}, the key step in all three scenarios above is to prove negative regularity,
that is, the solution $u$ lies in $L_x^2$ or better.  In scenarios~II and~III, the proof that $u\in L^2_x$ requires $d\geq5$; note that this expresses
better decay of the solution at spatial infinity.  Similar in spirit to \cite{KTV, Berbec, KVZ}, negative regularity is deduced
(via almost periodicity) from the minimality of the solution considered; recall that $u$ has minimal $L_t^\infty \dot H^{s_c}_x$ norm among
all blowup solutions.

A further manifestation of this minimality is the absence of a scattered wave at the endpoints of the lifespan $I$; more formally,
we have the following Duhamel formulae, which play an important role in proving negative regularity.  For a proof, see \cite[Section~6]{tvz:cc}
or \cite[Section~5]{Notes}.

\begin{lemma}[No waste Duhamel formulae]\label{L:duhamel}
Let $u$ be an almost periodic solution to \eqref{nls} on its maximal-lifespan $I$.  Then, for all $t\in I$,
\begin{equation}\label{Duhamel}
\begin{aligned}
u(t)&=\lim_{T\nearrow\,\sup I}i\int_t^T e^{i(t-t')\Delta} F(u(t'))\,dt'\\
&=-\lim_{T\searrow\,\inf I}i\int_T^t e^{i(t-t')\Delta} F(u(t'))\,dt',
\end{aligned}
\end{equation}
as weak limits in $\dot H^{s_c}_x$.
\end{lemma}

We preclude the finite-time blowup scenario in Section~\ref{S:self-similar}.  In order to achieve this, we use an argument based on the Strichartz
inequality and we take advantage of the finiteness of the blowup time.  We prove that in this case, the solution must have zero mass/energy.
This contradicts the fact that the solution blows up.

To preclude the remaining two scenarios, we follow closely the strategy in \cite{Berbec}.  As described above, the main point is to prove
additional decay/negative regularity; this is achieved in two steps.  First, we prove that the solution belongs to $L^\infty_t L^q_x$
for certain values of $q$ less than $dp/2$, which is the exponent obtained by applying Sobolev embedding to $\dot H^{s_c}_x$.  Just as in \cite{Berbec}, the proof
of this first step involves a bootstrap argument built off the Duhamel formulae \eqref{Duhamel}.  In order to disentangle
frequency interactions, we make use of an `acausal' Gronwall inequality, Lemma~\ref{L:Gronwall}. In the second step, we upgrade
this breach of scaling to negative regularity in $L^2_x$-based spaces.  To do this, we take advantage of the global existence
together with a double Duhamel trick in the spirit of \cite{Berbec, tao:attractor}.  In order to make the associated time
integrals converge, we need both $d\geq 5$ and the decay proved in step one.

In Section~\ref{S:cascade}, we use the negative regularity proved in Section~\ref{S:neg} together with the conservation of mass
to preclude the low-to-high frequency cascade.

In Section~\ref{S:Soliton}, we use the interaction Morawetz inequality to preclude the soliton.  In order to do this,
we need the negative regularity proved in Section~\ref{S:neg}.

\subsection*{Acknowledgements}
The first author was supported by NSF grant DMS-0701085.

%
%
%
%

\section{Notations and useful lemmas}\label{S:Notations}

\subsection{Some notation}
We write $X \lesssim Y$ or $Y \gtrsim X$ whenever $X \leq CY$ for some constant $C>0$.  We use $O(Y)$ to denote any quantity $X$
such that $|X| \lesssim Y$.  We use the notation $X \sim Y$ to mean $X \lesssim Y \lesssim X$.  The fact that these constants
depend upon the dimension $d$ or the power $p$ will be suppressed.  If $C$ depends upon some additional parameters, we will indicate this with
subscripts; for example, $X \lesssim_u Y$ denotes the assertion that $X \leq C_u Y$ for some $C_u$ depending on $u$; similarly
for $X \sim_u Y$, $X = O_u(Y)$, etc.  We denote by $X\pm$ any quantity of the form $X\pm\eps$ for any $\eps>0$.

For any spacetime slab $I\times \R^d$, we use $L_t^qL_x^r(I\times \R^d)$ to denote the Banach space of functions $u: I\times
\R^d\to \mathbb C$ whose norm is
$$
\|u\|_{L_t^qL_x^r(I\times\R^d)}:=\Bigl(\int_I\|u(t)\|_{L^r_x}^q \, dt\Bigr)^{\frac 1q}<\infty,
$$
with the usual modifications when $q$ or $r$ is equal to infinity.  When $q=r$ we abbreviate $L^q_t L^q_x$ as $L^q_{t,x}$.

We define the Fourier transform on $\R^d$ by
$$
\hat f(\xi):= (2\pi)^{-d/2} \int_{\R^d} e^{-ix\xi}f(x)\,dx.
$$
For $s\in \R$, we define the fractional differentiation/integral operator
$$
\widehat{|\nabla|^s f}(\xi):=|\xi|^s\hat f(\xi),
$$
which in turn defines the homogeneous Sobolev norm
$$
\|f\|_{\dot H_x^s(\R^d)}:=\bigl\||\nabla|^s f\bigr\|_{L_x^2(\R^d)}.
$$

Finally, we use the notation
$$
\nabla F(u(x)) := \nabla u(x) \cdot F'(u(x)):= \nabla u(x) F_z(u(x)) + \overline{\nabla u(x)} F_{\bar z}(u(x)),
$$
where $F_z$, $F_{\bar{z}}$ denote the usual complex derivatives
$$
F_z := \frac{1}{2}\Bigl(\frac{\partial F}{\partial x} - i\frac{\partial F}{\partial y}\Bigr), \quad F_{\bar{z}} :=
\frac{1}{2}\Bigl(\frac{\partial F}{\partial x} + i\frac{\partial F}{\partial y}\Bigr).
$$

\subsection{Strichartz estimates}
Let $e^{it\Delta}$ be the free Schr\"odinger evolution.  From the explicit formula
$$ e^{it\Delta} f(x) = \frac{1}{(4\pi i t)^{d/2}} \int_{\R^d} e^{i|x-y|^2/4t} f(y)\, dy,$$
one easily obtains the standard dispersive inequality
\begin{equation}\label{dispersive}
\| e^{it\Delta} f \|_{L_x^\infty(\R^d)} \lesssim|t|^{-\frac d2} \| f \|_{L_x^1(\R^d)}
\end{equation}
for all $t\neq 0$.  In particular, as the free propagator conserves the $L_x^2$-norm,
\begin{equation}\label{dispersive-p}
\| e^{it\Delta} f \|_{L_x^p(\R^d)} \lesssim|t|^{d(\frac 1p - \frac12)} \| f \|_{L_x^{p'}(\R^d)}
\end{equation}
for all $t\neq 0$ and $2\leq p\leq \infty$, where $\tfrac1p+\tfrac 1{p'}=1$.

\begin{definition}[Admissible pairs]
For $d\geq3$, we say that a pair of exponents $(q,r)$ is \emph{Schr\"odinger-admissible} if
\begin{equation}\label{admissible}
\frac 2q+\frac dr=\frac d2 \quad  \text{and} \quad  2\le q, r \le\infty.
\end{equation}
For a fixed spacetime slab $I\times \R^d$, we define the \emph{Strichartz norm}
$$
\|u\|_{S^0(I)}:=\sup_{(q,r)\ \text{admissible}} \|u\|_{L_t^qL_x^r(\ird)}.
$$
We write $S^0(I)$ for the closure of all test functions under this norm and denote by $N^0(I)$ the dual of $S^0(I)$.
\end{definition}

A simple application of Sobolev embedding yields
\begin{align*}
\bigl\||\nabla|^{s_c} u\bigr\|_{L_t^\infty L_x^2(\ird)} + \bigl\||\nabla|^{s_c} u\bigr\|_{L_{t,x}^{\frac{2(d+2)}d}(\ird)}
   & +\|u\|_{L_t^\infty L_x^{\frac {dp}2}(\ird)} + \|u\|_{L_{t,x}^{\frac{p(d+2)}2}(\ird)} \\
   &\lesssim \bigl\||\nabla|^{s_c} u\bigr\|_{S^0(I)}
\end{align*}
for all $d\geq 3$.

As a consequence of the dispersive estimate \eqref{dispersive}, we have the following standard Strichartz estimate.

\begin{lemma}[Strichartz]\label{L:Strichartz}
Let $s\geq 0$, let $I$ be a compact time interval, and let $u: I\times\R^d \to \mathbb C$ be a solution to the forced
Schr\"odinger equation
$$
iu_t+\Delta u=F.
$$
Then,
$$
\bigl\||\nabla|^s u\bigr\|_{ S^0(I)}\lesssim \bigl\||\nabla|^s u(t_0)\bigr\|_{L_x^2}+\bigl\||\nabla|^s F\bigr\|_{N^0(I)}
$$
for any $t_0\in I$.
\end{lemma}

\begin{proof}
See, for example, \cite{gv:strichartz, strichartz}.  For the endpoint $(q,r)=\bigl(2,\frac{2d}{d-2}\bigr)$ in dimensions $d\geq 3$, see \cite{tao:keel}.
\end{proof}

\subsection{Basic harmonic analysis}\label{ss:basic}
Let $\varphi(\xi)$ be a radial bump function supported in the ball $\{ \xi \in \R^d: |\xi| \leq \tfrac {11}{10} \}$ and equal to
$1$ on the ball $\{ \xi \in \R^d: |\xi| \leq 1 \}$.  For each number $N > 0$, we define the Fourier multipliers
\begin{align*}
\widehat{P_{\leq N} f}(\xi) &:=\widehat{f_{\leq N}}(\xi):= \varphi(\xi/N) \hat f(\xi)\\
\widehat{P_{> N} f}(\xi) &:=\widehat{f_{> N}}(\xi):= (1 - \varphi(\xi/N)) \hat f(\xi)\\
\widehat{P_N f}(\xi) &:= \widehat{f_{ N}}(\xi):= (\varphi(\xi/N) - \varphi(2\xi/N)) \hat f(\xi)
\end{align*}
and similarly $P_{<N}$ and $P_{\geq N}$.  We also define
$$ P_{M < \cdot \leq N} := P_{\leq N} - P_{\leq M} = \sum_{M < N' \leq N} P_{N'}$$
whenever $M < N$.  We will usually use these multipliers when $M$ and $N$ are \emph{dyadic numbers} (that is, of the form $2^n$
for some integer $n$); in particular, all summations over $N$ or $M$ are understood to be over dyadic numbers.  Nevertheless, it
will occasionally be convenient to allow $M$ and $N$ to not be a power of $2$.

Like all Fourier multipliers, the Littlewood-Paley operators commute with the propagator $e^{it\Delta}$, as well as with
differential operators such as $i\partial_t + \Delta$. We will use basic properties of these operators many many times,
including

\begin{lemma}[Bernstein estimates]\label{Bernstein}
 For $1 \leq r \leq q \leq \infty$,
\begin{align*}
\bigl\| |\nabla|^{\pm s} P_N f\bigr\|_{L^r_x(\R^d)} &\sim N^{\pm s} \| P_N f \|_{L^r_x(\R^d)},\\
\|P_{\leq N} f\|_{L^q_x(\R^d)} &\lesssim N^{\frac{d}{r}-\frac{d}{q}} \|P_{\leq N} f\|_{L^r_x(\R^d)},\\
\|P_N f\|_{L^q_x(\R^d)} &\lesssim N^{\frac{d}{r}-\frac{d}{q}} \| P_N f\|_{L^r_x(\R^d)}.
\end{align*}
\end{lemma}

\begin{lemma}[Product rule, \cite{ChW:fractional chain rule}]\label{L:product rule}
Let $s\in(0,1]$ and $1<r, r_1,r_2,q_1,q_2<\infty$ such that $\frac 1r =\frac 1{r_i}+\frac 1{q_i}$ for $i=1,2$.  Then,
$$
\bigl\||\nabla|^s(fg)\bigr\|_r
\lesssim \|f\|_{r_1}\bigl\||\nabla|^sg\bigr\|_{q_1} + \bigl\||\nabla|^sf\bigr\|_{r_2}\|g\|_{q_2}.
$$
\end{lemma}

We will also need the following fractional chain rule from \cite{ChW:fractional chain rule}.  For a textbook treatment, see
\cite[\S 2.4]{Taylor:book}.

\begin{lemma}[Fractional chain rule, \cite{ChW:fractional chain rule}]\label{L:Fract chain rule}
Suppose $G\in C^1(\mathbb C)$, $s \in (0,1]$, and $1<q,q_1,q_2<\infty$ are such that $\frac 1q=\frac 1{q_1}+\frac 1{q_2}$.   Then,
$$
\bigl\||\nabla|^sG(u)\bigr\|_q\lesssim \|G'(u)\|_{q_1}\bigl\||\nabla|^s u\bigr\|_{q_2}.
$$
\end{lemma}

When the function $G$ is no longer $C^1$, but merely H\"older continuous, we have the following chain rule:

\begin{lemma}[Fractional chain rule for a H\"older continuous function, \cite{Monica:thesis}]\label{L:FDFP}
\qquad Let $G$ be a H\"older continuous function of order $0<p<1$.  Then, for every $0<s<p$, $1<q<\infty$,
and $\tfrac{s}p<\sigma<1$ we have
\begin{align}\label{fdfp2}
\bigl\| |\nabla|^s G(u)\bigr\|_q
\lesssim \bigl\||u|^{p-\frac{s}{\sigma}}\bigr\|_{q_1} \bigl\||\nabla|^\sigma u\bigr\|^{\frac{s}{\sigma}}_{\frac{s}{\sigma}q_2},
\end{align}
provided $\tfrac{1}{q}=\tfrac{1}{q_1} +\tfrac{1}{q_2}$ and $(1-\frac s{p \sigma})q_1>1$.
\end{lemma}

As a direct consequence of the two fractional chain rules above and interpolation, we have the following

\begin{corollary}\label{C:s deriv}
Let $F(u)=|u|^p u$ and let $s\geq 0$ if $p$ is an even integer or $0\leq s<1+p$ otherwise.  Then, on any spacetime slab $I\times\R^d$ we have
\begin{align*}
\bigl\||\nabla|^{s} F(u)\bigr\|_{N^0(I)}
&\lesssim \bigl\||\nabla|^{s} u\bigr\|_{S^0(I)} \|u\|_{L_{t,x}^{\frac{p(d+2)}2}}^p
\end{align*}
and
\begin{align*}
\bigl\||\nabla|^{s} F(u)\bigr\|_{L_t^\infty L_x^{\frac{2d}{d+4}}}
&\lesssim \bigl\||\nabla|^{s} u\bigr\|_{L_t^\infty L_x^2} \|u\|_{L_t^\infty L_x^{dp/2}}^p.
\end{align*}
\end{corollary}

Revisiting the proof of Lemma~\ref{L:FDFP}, we obtain the following lemma.  In actually, the result can de deduced directly from \eqref{fdfp2}
and Lemma~\ref{L:product rule}; however, the result \eqref{E:pointwise} appearing in the proof will be needed in Section~\ref{S:Reduct}.

\begin{lemma}\label{L:frac deriv of diff}
Let $G$ be a H\"older continuous function of order $0<p\leq 1$ and let $0<s<\sigma p<p$.  For $1<q,q_1,q_2,r_1,r_2,r_3<\infty$ such that
$\frac 1q=\frac 1{q_1}+\frac 1{q_2}=\frac 1{r_1}+\frac 1{r_2}+\frac1{r_3}$ we have
\begin{align*}
\bigl\| |\nabla|^{s} \bigl[w &\cdot \bigl(G(u+v) - G(u) \bigr)  \bigr]\bigr\|_q\\
&\lesssim \bigl\| |\nabla|^{s} w \bigr\|_{q_1} \|v\|_{pq_2}^p + \|w\|_{r_1} \|v\|_{(p-\frac s\sigma)r_2}^{p-\frac s\sigma}
    \bigl( \bigl\| |\nabla|^\sigma v\bigr\|_{\frac s\sigma r_3}+ \bigl\| |\nabla|^\sigma u\bigr\|_{\frac s\sigma r_3}\bigr)^{\frac s\sigma},
\end{align*}
provided $(1-p)r_1, (p-\frac s\sigma)r_2>1$.
\end{lemma}

\begin{proof}
In \cite{Strich}, Strichartz proved that for all Schwartz functions $f$, $1<q<\infty$, and $0<s<1$,
\begin{equation*}
\bigl\||\nabla|^s f \bigr\|_{L^q_x} \sim \bigl\| \mathcal{D}_s(f) \bigr\|_{L^q_x},
\end{equation*}
where
\begin{equation}\label{Ds def}
\mathcal{D}_s(f)(x) := \biggl( \int_0^\infty \biggl|\int_{|y|<1} \bigl|f(x+ry)-f(x)\bigr| \,dy\biggr|^2 \frac{dr}{r^{1+2s}}\biggr)^{1/2}.
\end{equation}
In view of this, the claim will follow from the pointwise inequality
\begin{align}\label{E:pointwise}
\mathcal{D}_s\bigl(w \cdot \bigl[G(&u+v) - G(u) \bigr]  \bigr)\\
&\lesssim \mathcal D_s(w)|v|^p + \bigl[ M(|w|^{\frac1{1-p}})\bigr]^{1-p}
    \bigl[ M(|v|)\bigr]^{p-\frac s\sigma}\bigl[ \mathcal{D}_\sigma(u+v) + \mathcal{D}_\sigma(u)\bigr]^{\frac s\sigma},\notag
\end{align}
where $M$ denotes the Hardy--Littlewood maximal function.

As $G$ is H\"older continuous of order $p$,
\begin{align}\label{long bound}
\bigl| \bigl(w & \cdot \bigl[G(u+v) - G(u) \bigr]\bigr)(x+ry) - \bigl(w \cdot \bigl[G(u+v) - G(u) \bigr]\bigr)(x) \bigr| \notag \\
&\leq \bigl| w(x+ry) - w(x)\bigr| \bigl|G(u+v)(x) - G(u)(x)  \bigr|\notag\\
&\quad + \bigl| w(x+ry)\bigr|\bigl| \bigl[G(u+v) - G(u) \bigr](x+ry) - \bigl[G(u+v) - G(u) \bigr](x)\bigr|\notag\\
&\lesssim \bigl| w(x+ry) - w(x)\bigr| |v(x)|^p  + H(x)
\end{align}
with
\begin{align}
H:= \bigl| w(x+ry)\bigr|\bigl| \bigl[G(u+v) - G(u) \bigr](x+ry) - \bigl[G(u+v) - G(u) \bigr](x)\bigr|.
\end{align}

Note that the first term on the right-hand side of \eqref{long bound} gives rise to the term $\mathcal D_s(w)|v|^p$ in \eqref{E:pointwise}.
Hence it remains to estimate the contribution of $H$.  In order to achieve this, we estimate $H$ in two different ways:
\begin{align}\label{big r}
H\lesssim \bigl| w(x+ry) \bigr|\Bigl[ |v(x+ry)|^p + |v(x)|^p \Bigr]
\end{align}
and also
\begin{align}\label{small r}
H\lesssim \bigl| w(x+ry)\bigr| \Bigl[\bigl|(u+v)(x+ ry) - (u+v)(x)\bigr|^p +\bigl|u(x+ ry) - u(x)\bigr|^p \Bigr].
\end{align}

Using H\"older's inequality and \eqref{big r}, we see that
\begin{align*}
\int_{A(x)}^\infty \biggl|\int_{|y|<1} H\,dy\biggr|^2 \frac{dr}{r^{1+2s}}
&\lesssim \bigl[ M(|w|^{\frac1{1-p}})(x)\bigr]^{2(1-p)} \bigl[ M(|v|)(x)\bigr]^{2p} \int_{A(x)}^\infty \frac{dr}{r^{1+2s}} \\
&\lesssim \bigl[A(x)\bigr]^{-2s}\bigl[ M(|w|^{\frac1{1-p}})(x)\bigr]^{2(1-p)} \bigl[ M(|v|)(x)\bigr]^{2p}.
\end{align*}
The precise value of $A(x)$ will be determined below.

We now turn our attention to small values of $r$.  Using \eqref{small r} together with H\"older's inequality, we see that
\begin{align*}
\int_{|y|<1}\!\!H\,dy
&\lesssim \bigl[M\bigl(|w|^{\frac{1}{1-p}}\bigr)(x)\bigr]^{1-p} \biggl[ \int_{|y|<1} \bigl|(u+v)(x+ ry) - (u+v)(x)\bigr| \\
&\qquad \qquad \qquad \qquad \qquad \qquad + \bigl|u(x+ ry) - u(x)\bigr| \ dy \biggr]^p
\end{align*}
and so, applying H\"older's inequality again, we find
\begin{align*}
\int_0^{A(x)} \biggl|\int_{|y|<1} &H\, dy\biggr|^2 \frac{dr}{r^{1+2s}} \\
&\lesssim \bigl[M\bigl(|w|^{\frac{1}{1-p}}\bigr)(x)\bigr]^{2(1-p)}
    \bigl[A(x)\bigr]^{2(\sigma p -s)} \bigl[ \mathcal{D}_\sigma(u+v)(x) + \mathcal{D}_\sigma(u)(x)\bigr]^{2p}.
\end{align*}

Putting things together and optimizing the choice of $A(x)$, we derive \eqref{E:pointwise}.  This finishes the proof of the lemma.
\end{proof}

The next result is formally similar to Lemma~\ref{L:FDFP}.  The proof is simple; see the appendix in \cite{Notes}.  It is used in the proof of
Lemma~\ref{L:recurrence}.

\begin{lemma}[Nonlinear Bernstein, \cite{Notes}]\label{L:Nonlin Bernstein}
Let $G:\C\to\C$ be H\"older continuous of order $0<p\leq 1$. Then
$$
\|P_N G(u) \|_{L^{q/p}_x(\R^d)} \lesssim N^{-p} \|\nabla u\|_{L^q_x(\R^d)}^p
$$
for any $1\leq q < \infty$.
\end{lemma}

\subsection{Concentration compactness}\label{SS:cc}
In this subsection we record the linear profile decomposition statement which will lead to the reduction in Theorem~\ref{T:reduct}.
We first recall the symmetries of the equation \eqref{nls} which fix the initial surface $t=0$.

\begin{definition}[Symmetry group]\label{D:sym}
For any phase $\theta \in \R/2\pi \Z$, position $x_0\in \R^d$, and scaling parameter $\lambda > 0$, we define a unitary
transformation $g_{\theta,x_0,\lambda}: \dot H^{s_c}_x(\R^d) \to \dot H^{s_c}_x(\R^d)$ by
$$
[g_{\theta,x_0, \lambda} f](x) :=  \lambda^{-\frac2p} e^{i\theta}  f\bigl( \lambda^{-1}(x-x_0) \bigr).
$$
Recall that $s_c:=\tfrac d2-\tfrac 2p$.  Let $G$ denote the collection of such transformations.  For a function $u: I \times \R^d \to \C$,
we define $T_{g_{\theta,x_0,\lambda}} u: \lambda^2 I \times \R^d \to \C$ where $\lambda^2 I := \{ \lambda^2 t: t \in I \}$ by the formula
$$
[T_{g_{\theta,x_0, \lambda}} u](t,x) :=  \lambda^{-\frac2p} e^{i\theta} u\bigl( \lambda^{-2}t, \lambda^{-1}(x-x_0)\bigr).
$$
Note that if $u$ is a solution to \eqref{nls}, then $T_{g}u$ is a solution to \eqref{nls} with initial data $g u_0$.
\end{definition}

\begin{remark} It is easy to verify that $G$ is a group and that the map $g \mapsto T_g$ is a homomorphism.
The map $u \mapsto T_g u$ maps solutions to \eqref{nls} to solutions with the same scattering size as $u$, that is, $S(T_g u)=S(u)$.
Furthermore, $u$ is a maximal-lifespan solution if and only if $T_g u$ is a maximal-lifespan solution.
\end{remark}

We are now ready to state the linear profile decomposition; in the generality needed here, this was proved in \cite{Shao:profile}.
For $s_c=0$ the linear profile decomposition was proved in \cite{BegoutVargas, carles-keraani, merle-vega}, while for $s_c=1$ it
was established in \cite{keraani-h1}.

\begin{lemma} [Linear profile decomposition, \cite{Shao:profile}]\label{L:cc}
Fix $s_c>0$ and let $\{u_n\}_{n\geq 1}$ be a sequence of functions bounded in $\dot H_x^{s_c}(\R^d)$. Then, after passing to a
subsequence if necessary, there exist functions $\{\phi^j\}_{j\geq 1}\subset \dot H_x^{s_c}(\R^d)$, group elements
$\gnj \in G$, and times $\tnj\in \R$ such that for all $J\geq 1$ we have the decomposition
\begin{align*}
u_n = \sum_{j=1}^J \gnj e^{i\tnj\Delta}\phi^j + \wnJ
\end{align*}
with the following properties:
\begin{itemize}
\item $\wnJ \in \dot H^{s_c}_x(\R^d)$ and obey
\begin{equation*}
\lim_{J\to \infty}\limsup_{n\to\infty} \bigl\| e^{it\Delta}\wnJ \bigr\|_{L_{t,x}^{\frac{p(d+2)}2}(\R\times\R^d)}=0.
\end{equation*}
\item For any $j \neq j'$,
\begin{align*}
\frac{\lambda_n^j}{\lambda_n^{j'}} + \frac{\lambda_n^{j'}}{\lambda_n^{j}}
    + \frac{|x_n^j-x_n^{j'}|^2}{\lambda_n^j \lambda_n^{j'}}
    + \frac{\bigl|t_n^j(\lambda_n^j)^2- t_n^{j'}(\lambda_n^{j'})^2\bigr|}{\lambda_n^j\lambda_n^{j'}}\to\infty
    \quad \text{as } n\to \infty.
\end{align*}
\item For any $J \geq 1$ we have the decoupling properties:
\begin{equation*}
\lim_{n \to \infty} \Bigl[ \bigl\||\nabla|^{s_c} u_n\bigr\|_2^2 - \sum_{j=1}^J \bigl\||\nabla|^{s_c} \phi^j \bigr\|_2^2
            - \bigl\||\nabla|^{s_c} \wnJ\bigr\|^2_2 \Bigr] = 0
\end{equation*}
and for any $1\leq j\leq J$,
\begin{align*}
e^{-i\tnj\Delta}[(\gnj)^{-1}\wnJ] \to 0 \quad \text{weakly in } \dot H_x^{s_c} \text{ as } n\to \infty.
\end{align*}
\end{itemize}
\end{lemma}

\subsection{A Gronwall inequality}

Our last technical tool is a form of Gronwall's inequality that involves both the past and the future, `acausal' in the terminology of \cite{tao:book}.
We import it from \cite{Berbec}, where it was used for precisely the same purpose as it will be here.

\begin{lemma}[Acausal Gronwall inequality, \cite{Berbec}]\label{L:Gronwall}
Given $\gamma>0$, $0<\eta<\tfrac12(1-2^{-\gamma})$, and $\{b_k\}\in\ell^\infty(\Z^+)$, let $x_k\in\ell^\infty(\Z^+)$ be a
non-negative sequence obeying
\begin{align*}
x_k \leq b_k + \eta \sum_{l=0}^\infty 2^{-\gamma|k-l|} x_l \qquad \text{for all $k\geq 0$.}
\end{align*}
Then
\begin{align*}
x_k \lesssim \sum_{l=0}^{k} r^{|k-l|} b_l  \qquad \text{for all $k\geq 0$}
\end{align*}
for some $r=r(\eta)\in (2^{-\gamma},1)$.  Moreover, $r\downarrow 2^{-\gamma}$ as $\eta\downarrow 0$.
\end{lemma}

%
%
%
%

\section{Local well-posedness}\label{S:Local theory}

In this section we develop the local well-posedness theory for \eqref{nls}.  The arguments we use are inspired by previous work
on nonlinear Schr\"odinger equations at critical regularity.  For $s_c\in[0,1]$ the standard local well-posedness theory
(see Theorem~\ref{T:standard lwp} below) was established by Cazenave and Weissler, \cite{cwI}; see also \cite{cazenave, tao:book}.
For stability results (see Theorems~\ref{T:stab cubic} and \ref{T:stab} below) in the mass- and energy-critical settings (i.e. $s_c=0,1$), see
\cite{ckstt:gwp, RV, TV, tvz:cc}.  In this section, we mainly follow the exposition in \cite{Notes}, which revisits the local
theory for the mass- and energy-critical NLS.

We start with the following standard local well-posedness result, for which one assumes that the initial data lies in the inhomogeneous
critical Sobolev space.  This assumption simplifies the proof and can be removed a posteriori by using the stability results proved below.

\begin{theorem}[Standard local well-posedness]\label{T:standard lwp}
Let $d\geq 1$, $s_c\geq 0$, and let $u_0 \in H^{s_c}_x(\R^d)$.  Assume in addition that $s_c<1+p$ if $p$ is not an even integer.
Then there exists $\eta_0=\eta_0(d)>0$ such that if $0<\eta \leq \eta_0$ and $I$ is a compact interval containing zero such that
\begin{align}\label{LT:small free evol}
\bigl\| |\nabla|^{s_c} e^{it\Delta} u_0\bigr\|_{L_t^{p+2} L_x^{\frac{2d(p+2)}{2(d-2)+dp}} (I\times\R^d)} \leq \eta,
\end{align}
then there exists a unique solution $u$ to \eqref{nls} on $I\times\R^d$.  Moreover, we have the bounds
\begin{align}
\bigl\| |\nabla|^{s_c} u \bigr\|_{L_t^{p+2} L_x^{\frac{2d(p+2)}{2(d-2)+dp}} (I\times\R^d)} &\leq 2\eta \label{LT:small sol u}\\
\bigl\| |\nabla|^{s_c} u \bigr\|_{S^0 (I\times\R^d)} &\lesssim \bigl\| |\nabla|^{s_c} u_0 \bigr\|_{L_x^2} + \eta^{1+p} \label{LT:bdd sc}\\
\| u \|_{S^0 (I\times\R^d)} &\lesssim \| u_0\|_{L_x^2}. \label{LT:bdd s0}
\end{align}
\end{theorem}

\begin{remark}
By the Strichartz inequality, we know that
$$
\bigl\| |\nabla|^{s_c} e^{it\Delta} u_0\bigr\|_{L_t^{p+2} L_x^{\frac{2d(p+2)}{2(d-2)+dp}} (\R\times\R^d)}
\lesssim \bigl \||\nabla|^{s_c} u_0 \bigr\|_{L_x^2}.
$$
Thus, \eqref{LT:small free evol} holds with $I=\R$ for initial data with sufficiently small norm.  Alternatively,
by the monotone convergence theorem, \eqref{LT:small free evol} holds provided $I$ is chosen sufficiently small.
Note that by scaling, the length of the interval $I$ depends on the fine properties of $u_0$, not only on its norm.
\end{remark}

\begin{proof}
We will essentially repeat the standard argument from \cite{cwI}; the fractional chain rule Lemma~\ref{L:Fract chain rule}
leads to some simplifications.

The theorem follows from a contraction mapping argument.  More precisely, using the Strichartz estimates from
Lemma~\ref{L:Strichartz}, we will show that the map $u\mapsto \Phi(u)$ defined by
$$
\Phi(u)(t):= e^{it\Delta} u_0 - i \int_{0}^{t} e^{i(t-s)\Delta} F(u(s))\, ds,
$$
is a contraction on the set $B_1\cap B_2$ where
\begin{align*}
B_1&:=\Bigl\{u\in L_t^\infty H^{s_c}_x (I\times\R^d): \, \|u\|_{L_t^\infty H^{s_c}_x(I\times\R^d)} \leq 2 \|u_0\|_{H^{s_c}_x} + C(d,p)(2\eta)^{1+p}\Bigr\}\\
B_2&:=\Bigl\{u\in L_t^{p+2}  W_x^{s_c, \frac{2d(p+2)}{2(d-2)+dp}}(I\times\R^d): \,
    \bigl\| |\nabla|^{s_c} u \bigr\|_{L_t^{p+2} L_x^{\frac{2d(p+2)}{2(d-2)+dp}} (I\times\R^d)}\! \! \leq \! 2\eta \\
&\qquad\qquad\qquad\qquad\qquad\qquad\qquad \text{and} \quad
        \bigl\| u \bigr\|_{L_t^{p+2} L_x^{\frac{2d(p+2)}{2(d-2)+dp}} (I\times\R^d)}\! \! \leq \! 2C(d,p)\|u_0\|_{L_x^2}\Bigr\}
\end{align*}
under the metric given by
$$
d(u,v):= \| u -v \|_{L_t^{p+2} L_x^{\frac{2d(p+2)}{2(d-2)+dp}}(I\times\R^d)}.
$$
Here $C(d,p)$ denotes a constant that changes from line to line.  Note that the norm appearing in the metric scales like
$L^2_x$.  Note also that both $B_1$ and $B_2$ are closed (and hence complete) in this metric.

Using the Strichartz inequality followed by Corollary~\ref{C:s deriv} and Sobolev embedding, we find that for $u\in B_1\cap B_2$,
\begin{align*}
\|\Phi(u)\|_{L_t^\infty H^{s_c}_x(I\times\R^d)}
&\leq \|u_0\|_{H^{s_c}_x} + C(d,p)\bigl\| \langle\nabla\rangle^{s_c} F(u)\bigr\|_{L_t^{\frac{p+2}{p+1}} L_x^{\frac{2d(p+2)}{2(d+2)+dp}}(I\times\R^d)}\\
&\leq \|u_0\|_{H^{s_c}_x} + C(d,p)\bigl\| \langle\nabla\rangle^{s_c} u \bigr\|_{L_t^{p+2} L_x^{\frac{2d(p+2)}{2(d-2)+dp}}(I\times\R^d)}
            \|u\|_{L_t^{p+2} L_x^{\frac{dp(p+2)}4}(I\times\R^d)}^p\\
&\leq \|u_0\|_{H^{s_c}_x} + C(d,p)\bigl(2\eta + 2C(d,p) \|u_0\|_{L_x^2} \bigr)
            \bigl\| |\nabla|^{s_c} u \bigr\|_{L_t^{p+2} L_x^{\frac{2d(p+2)}{2(d-2)+dp}} (I\times\R^d)}^p\\
&\leq \|u_0\|_{H^{s_c}_x} + C(d,p) \bigl(2\eta + 2C(d,p) \|u_0\|_{L_x^2} \bigr)(2\eta)^p
\end{align*}
and similarly,
\begin{align*}
\bigl\| \Phi(u) \bigr\|_{L_t^{p+2} L_x^{\frac{2d(p+2)}{2(d-2)+dp}}  (I\times\R^d)}
&\leq C(d,p) \|u_0\|_{L_x^2} + C(d,p)\bigl\|F(u)\bigr\|_{L_t^{\frac{p+2}{p+1}} L_x^{\frac{2d(p+2)}{2(d+2)+dp}}(I\times\R^d)}\\
&\leq C(d,p) \|u_0\|_{L_x^2} + C(d,p) \|u_0\|_{L_x^2} (2\eta)^p.
\end{align*}
Arguing as above and invoking \eqref{LT:small free evol}, we obtain
\begin{align*}
\bigl\| |\nabla|^{s_c} \Phi(u) \bigr\|_{L_t^{p+2} L_x^{\frac{2d(p+2)}{2(d-2)+dp}} (I\times\R^d)}
&\leq \eta + C(d,p)\bigl\|  |\nabla|^{s_c} F(u)\bigr\|_{L_t^{\frac{p+2}{p+1}} L_x^{\frac{2d(p+2)}{2(d+2)+dp}}(I\times\R^d)}\\
&\leq \eta + C(d,p)(2\eta)^{1+p}.
\end{align*}
Thus, choosing $\eta_0=\eta_0(d)$ sufficiently small, we see that for $0<\eta\leq \eta_0$, the functional $\Phi$ maps the set
$B_1\cap B_2$ back to itself.  To see that $\Phi$ is a contraction, we repeat the computations above to obtain
\begin{align*}
\bigl\| \Phi(u) - \Phi(v) \bigr\|_{L_t^{p+2} L_x^{\frac{2d(p+2)}{2(d-2)+dp}}(I\times\R^d)}
&\leq C(d,p) \bigl\| F(u) - F(v)\bigr\|_{L_t^{\frac{p+2}{p+1}} L_x^{\frac{2d(p+2)}{2(d+2)+dp}}(I\times\R^d)}\\
&\leq C(d,p) (2\eta)^p \|u-v\|_{L_t^{p+2} L_x^{\frac{2d(p+2)}{2(d-2)+dp}}(I\times\R^d)}.
\end{align*}
Thus, choosing $\eta_0=\eta_0(d)$ even smaller (if necessary), we can guarantee that $\Phi$ is a contraction on the set $B_1\cap B_2$.
By the contraction mapping theorem, it follows that $\Phi$ has a fixed point in $B_1\cap B_2$.  Moreover, noting that $\Phi$ maps into
$C_t^0 H^{s_c}_x$ (not just $L_t^\infty H^{s_c}_x$), we derive (after one more application of the Strichartz inequality)
that the fixed point of $\Phi$ is indeed a solution to \eqref{nls}.

We now turn our attention to the uniqueness statement.  Since uniqueness is a local property, it suffices to study a neighbourhood of $t=0$.
By Definition~\ref{D:solution} (and the Strichartz inequality), any solution to \eqref{nls} belongs to $B_1\cap B_2$ on some such neighbourhood.
Uniqueness thus follows from uniqueness in the contraction mapping theorem.

The claims \eqref{LT:bdd sc} and \eqref{LT:bdd s0} follow from another application of the Strichartz inequality, as above.
\end{proof}

Next, we will establish a stability theory for \eqref{nls} in the settings of Theorems~\ref{T:main cubic} and~\ref{T:main}.
We start with the cubic NLS.

\begin{theorem}[Stability -- the cubic]\label{T:stab cubic}
Let $d\geq 2$ and let $I$ a compact time interval containing zero and $\tilde u$ be an approximate solution to \eqref{nls} on $I\times \R^d$
in the sense that
$$
i\tilde u_t =-\Delta \tilde u+ |\tilde u|^2\tilde u+ e
$$
for some function $e$.  Assume that
\begin{align}
\|\tilde u\|_{L_t^\infty \dot H_x^{\frac{d-2}2}(I\times \R^d)}&\le E \label{LT:finite energy cubic}\\
S_I(\tilde u)&\le L \label{LT: finite s-t cubic}
\end{align}
for some positive constants $E$ and $L$.  Let $u_0\in \dot H_x^{\frac{d-2}2}$ and assume the smallness conditions
\begin{align}
\|u_0-\tilde u_0\|_{\dot H_x^{\frac{d-2}2}}&\leq \eps \label{LT:diff energy cubic}\\
\bigl\| |\nabla|^{\frac{d-2}2} e \bigr\|_{N^0(I)}&\leq \eps \label{LT:error small cubic}
\end{align}
for some $0<\eps<\eps_1=\eps_1(E,L)$. Then, there exists a unique strong solution $u:I\times\R^d\mapsto \C$ to \eqref{nls} with initial data
$u_0$ at time $t=0$ satisfying
\begin{align}
S_I(u-\tilde u) &\leq C(E,L)\eps^{d+2} \label{LT:close in s-t cubic}\\
\bigl\| |\nabla|^{\frac{d-2}2} (u-\tilde u)\bigr\|_{S^0(I)} &\leq C(E,L)\eps \label{LT:close in S cubic}\\
\bigl\| |\nabla|^{\frac{d-2}2} u \bigr\|_{S^0(I)} &\leq C(E,L) \label{LT:u in Sc cubic}.
\end{align}
\end{theorem}

\begin{proof}
We will prove the theorem under the additional assumption that $u_0\in L_x^2$, so that we can rely on Theorem~\ref{T:standard lwp}
to guarantee that $u$ exists.  This additional assumption can be removed \emph{a posteriori}
by the usual limiting argument: approximate $u_0$ in $\dot H^{s_c}_x$ by $\{u_n(0)\}_n\subseteq H^{s_c}_x$ and apply the theorem with $\tilde u = u_m$,
$u=u_n$, and $e=0$ to deduce that the sequence of solutions $\{u_n\}_n$ with initial data $\{u_n(0)\}_n$ is Cauchy in critical norms
and thus convergent to a solution $u$ with initial data $u_0$ which obeys $|\nabla|^{s_c} u \in S^0(I)$.  Thus, it suffices to prove
\eqref{LT:close in s-t cubic} through \eqref{LT:u in Sc cubic} as \emph{a priori} estimates, that is we assume that the solution $u$
exists and obeys $|\nabla|^{s_c} u \in S^0(I)$.

We first prove \eqref{LT:close in s-t cubic} through \eqref{LT:u in Sc cubic} under the stronger hypothesis that
\begin{align}\label{LT: tilde u small}
\bigl\| |\nabla|^{\frac{d-2}2} \tilde u \bigr\|_{L_t^{d+2} L_x^{\frac{2d(d+2)}{d^2+2d-4}}(I\times\R^d)}\leq \delta
\end{align}
for some small $0<\delta=\delta(d,p)$.

Let $w:=u-\tilde u$.  Then $w$ satisfies the following initial value problem
\begin{equation*}
\begin{cases}
iw_t=-\Delta w + F(\tilde u+w)-F(\tilde u) -e\\
w(0)=u_0-\tilde u_0.
\end{cases}
\end{equation*}
For $t\in I$ we define
$$
A(t):=\bigl\|  |\nabla|^{\frac{d-2}2} \bigl[F(\tilde u+w)-F(\tilde u)\bigr] \bigr\|_{N^0([0,t])}.
$$
By the (fractional) chain rule and \eqref{LT: tilde u small},
\begin{align}\label{S(t)}
A(t)&\lesssim \bigl\| |\nabla|^{\frac{d-2}2} w\bigr\|_{S^0(I)}^3 + \delta \bigl\| |\nabla|^{\frac{d-2}2} w\bigr\|_{S^0(I)}^2
    +\delta^2 \bigl\| |\nabla|^{\frac{d-2}2} w\bigr\|_{S^0(I)}.
\end{align}
On the other hand, by Strichartz, \eqref{LT:diff energy cubic}, and \eqref{LT:error small cubic}, we get
\begin{align}\label{z}
\bigl\| |\nabla|^{\frac{d-2}2} w\bigr\|_{S^0(I)}
&\lesssim \|u_0-\tilde u_0\|_{\dot H_x^{\frac{d-2}2}} + A(t) + \bigl\| |\nabla|^{\frac{d-2}2} e \bigr\|_{N^0(I)}
\lesssim A(t)+\eps.
\end{align}
Combining \eqref{S(t)} and \eqref{z}, we obtain
$$
A(t)\lesssim (A(t)+\eps)^3+\delta(A(t)+\eps)^2+ \delta^2(A(t)+\eps)+\eps.
$$
A standard continuity argument then shows that if $\delta$ is taken sufficiently small,
$$
A(t)\lesssim \eps \ \ \text{for any}\ \ t\in I,
$$
which immediately implies \eqref{LT:close in s-t cubic} through \eqref{LT:u in Sc cubic} via an application of the Strichartz inequality
and the triangle inequality.

We now prove \eqref{LT:close in s-t cubic} through \eqref{LT:u in Sc cubic} under the hypothesis \eqref{LT: finite s-t cubic}, as opposed
to \eqref{LT: tilde u small}.  We first show that
\begin{equation}\label{LT:tilde u cubic}
\bigl\| |\nabla|^{\frac{d-2}2} \tilde u\bigr\|_{S^0(I)}\le C(E,L).
\end{equation}
Indeed, by \eqref{LT: finite s-t cubic} we may divide $I$ into $J_0=J_0(L, \eta)$ subintervals $I_j=[t_j,t_{j+1}]$
such that on each spacetime slab $I_j\times\R^d$
$$
\|\tilde u\|_{L_{t,x}^{d+2}(I_j\times\R^d)}\le \eta
$$
for a small constant $\eta>0$ to be chosen in a moment.  By the Strichartz inequality combined with Corollary~\ref{C:s deriv},
\eqref{LT:finite energy cubic}, and \eqref{LT:error small cubic},
\begin{align*}
\bigl\| |\nabla|^{\frac{d-2}2} \tilde u\bigr\|_{S^0(I_j)}
&\lesssim \|\tilde u(t_j)\|_{\dot H_x^{\frac{d-2}2}} + \bigl\| |\nabla|^{\frac{d-2}2} e \bigr\|_{ N^0(I_j)}
    + \bigl\| |\nabla|^{\frac{d-2}2} F(\tilde u)\bigr\|_{N^0(I_j)}\\
&\lesssim E + \eps + \|\tilde u\|_{L_{t,x}^{d+2}(I_j\times\R^d)}^2 \bigl\| |\nabla|^{\frac{d-2}2} \tilde u\bigr\|_{S^0(I_j)}\\
&\lesssim E + \eps + \eta^2 \bigl\| |\nabla|^{\frac{d-2}2} \tilde u\bigr\|_{S^0(I_j)}.
\end{align*}
Thus, choosing $\eta>0$ small depending on $d$ and $\eps_1$ sufficiently small depending on $E$, we obtain
$$
\bigl\| |\nabla|^{\frac{d-2}2} \tilde u\bigr\|_{S^0(I_j)}\lesssim E.
$$
Summing this over all subintervals $I_j$, we derive \eqref{LT:tilde u cubic}.  Thus, we may divide $I$ into $J_1=J_1(E,L)$ subintervals
$I_j=[t_j,t_{j+1}]$ such that on each spacetime slab $I_j\times\R^d$
$$
\bigl\| |\nabla|^{\frac{d-2}2} \tilde u \bigr\|_{L_t^{d+2} L_x^{\frac{2d(d+2)}{d^2+2d-4}}(I_j\times\R^d)}\le \delta
$$
for some small $\delta=\delta(d,p)>0$ as appearing in \eqref{LT: tilde u small}.

Choosing $\eps_1$ sufficiently small (depending on $J_1$), we can iterate the argument above to obtain for each $0\leq j< J_1$
and all $0<\eps<\eps_1$,
\begin{equation}\label{bounds on j cubic}
\begin{aligned}
S_{I_j}(u-\tilde u) &\leq C(j)\eps^{d+2}\\
\bigl\| |\nabla|^{\frac{d-2}2}( u-\tilde u) \bigr\|_{S^0(I_j)}&\leq C(j) \eps\\
\bigl\| |\nabla|^{\frac{d-2}2} u \bigr\|_{S^0(I_j)}&\leq C(j) E\\
\bigl\| |\nabla|^{\frac{d-2}2}\bigl[F(u)-F(\tilde u)\bigr] \bigr\|_{N^0(I_j)}&\leq C(j)\eps,
\end{aligned}
\end{equation}
provided we can show
\begin{align}\label{LT:left cubic}
\|u(t_j)-\tilde u(t_j)\|_{\dot H^{\frac{d-2}2}_x}&\leq C(j-1) \eps
\end{align}
for each $1\leq j<J_1$.  By the Strichartz inequality and the inductive hypothesis,
\begin{align*}
\|u(t_{j})- \tilde u(t_{j})\|_{\dot{H}^{\frac{d-2}2}_x}
&\lesssim \|u_0-\tilde u_0\|_{\dot{H}^{\frac{d-2}2}_x} + \bigl\| |\nabla|^{\frac{d-2}2} e \bigr\|_{N^0([0, t_{j}])}\\
&\quad + \bigl\| |\nabla|^{\frac{d-2}2} \bigl[ F(u)-F(\tilde u) \bigr]\bigr\|_{N^0([0, t_j])}\\
&\lesssim \eps +\sum_{k=0}^{j-1}C(k) \eps,
\end{align*}
which proves \eqref{LT:left cubic}.

Summing the bounds in \eqref{bounds on j cubic} over all subintervals $I_j$, we derive \eqref{LT:close in s-t cubic}
through \eqref{LT:u in Sc cubic}.  This completes the proof of the theorem.
\end{proof}

We now address the stability question in the setting of Theorem~\ref{T:main}.  We will prove the following result.

\begin{theorem}[Stability]\label{T:stab}
Let $d\geq 5$ and assume the critical regularity $s_c$ satisfies \eqref{sc}.  Let $I$ a compact time interval containing zero and let $\tilde u$
be an approximate solution to \eqref{nls} on $I\times \R^d$ in the sense that
$$
i\tilde u_t =-\Delta \tilde u + F(\tilde u)+ e
$$
for some function $e$.  Assume that
\begin{align}
\|\tilde u\|_{L_t^\infty \dot H_x^{s_c}(I\times \R^d)}&\le E \label{LT:finite energy}\\
S_I(\tilde u)&\le L \label{LT: finite s-t}
\end{align}
for some positive constants $E$ and $L$.  Let $u_0\in \dot H_x^{s_c}$ and assume the smallness conditions
\begin{align}
\|u_0-\tilde u_0\|_{\dot H_x^{s_c}}&\le \eps \label{LT:diff energy}\\
\bigl\| |\nabla|^{s_c} e \bigr\|_{N^0(I)}&\le \eps \label{LT:error small}
\end{align}
for some $0<\eps<\eps_1=\eps_1(E,L)$. Then, there exists a unique strong solution $u:I\times\R^d\mapsto \C$ to \eqref{nls} with initial data
$u_0$ at time $t=0$ satisfying
\begin{align}
S_I(u-\tilde u) &\leq C(E,L)\eps^{c_1} \label{LT:close in s-t}\\
\bigl\| |\nabla|^{s_c} (u-\tilde u)\bigr\|_{S^0(I)} &\leq C(E,L)\eps^{c_2} \label{LT:close in S}\\
\bigl\| |\nabla|^{s_c} u \bigr\|_{S^0(I)} &\leq C(E,L) \label{LT:u in Sc},
\end{align}
where $c_1, c_2$ are positive constants that depend on $d,p,E,$ and $L$.
\end{theorem}

\begin{remarks}
1. Theorems~\ref{T:stab cubic} and \ref{T:stab} imply the existence and uniqueness of maximal-lifespan solutions in Theorem~\ref{T:standard lwp}.
They also prove that the solutions depend uniformly continuously on the initial data (on bounded sets) in norms which are critical
with respect to scaling.  As a consequence, one can remove from Theorem~\ref{T:standard lwp} the assumption that the initial data belongs to $L_x^2$,
since every $\dot H^{s_c}_x$ function is well approximated by $H^{s_c}_x$ functions.

2. Using Theorem~\ref{T:standard lwp} (without the additional assumption that $u_0\in L_x^2$, due to the first point above), as well
as its proof, one easily derives Theorem~\ref{T:local}.  We omit the standard details.
\end{remarks}

We now turn to the proof of Theorem~\ref{T:stab}; the argument we present is inspired by the one used in the energy-critical setting
\cite{Notes, TV}; see also \cite{Nakanishi} for a similar technique in the context of the Klein--Gordon equation.
The idea is to work in spaces which are critical with respect to scaling but have a small fractional number of derivatives.

For the remainder of this subsection, for any time interval $I$ we will use the abbreviations
\begin{equation}\label{spaces}
\begin{aligned}
\|u\|_{X^0(I)}&:=\| u\|_{L_t^{q_0}L_x^{\frac{r_0d}{d-r_0s_c}}(I\times\R^d)}\\
\|u\|_{X(I)}&:=\bigl\| |\nabla |^{p/2} u\bigr\|_{L_t^{q_0}L_x^{\frac{2r_0d}{2d-r_0(2s_c-p)}}(I\times\R^d)}\\
\|F\|_{Y(I)}&:=\bigl\| |\nabla|^{p/2} F \bigr\|_{L_t^{\frac{q_0}{1+p}} L_x^{\frac{2r_0d}{2(1+p)(d-r_0s_c)+r_0p}}(I\times\R^d)},
\end{aligned}
\end{equation}
where $(q_0,r_0)=\bigl(\tfrac{2p(2+p)}{p^2-p(d-2)+4}, \tfrac{d(2+p)}{d-p+s_c(2+p)} \bigr)$ is a Schr\"odinger admissible pair.  Note that
because of \eqref{sc} we have $2<r_0<\tfrac{d}{s_c}$ and $\tfrac{p(d+2)}2<q_0<\infty$.

First, we connect the spaces in which the solution to \eqref{nls} is measured to the spaces
in which the nonlinearity is measured.  As usual, this is done via a Strichartz inequality; we reproduce the standard proof.

\begin{lemma}[Strichartz estimate]\label{LT:L:exotic strichartz}
Let $I$ be a compact time interval containing $t_0$.  Then
\begin{equation*}
\Bigl\|\int_{t_0}^t e^{i(t-s)\Delta}F(s)\,ds\Bigr\|_{X(I)}\lesssim \|F\|_{Y(I)}.
\end{equation*}
\end{lemma}

\begin{proof}
As $\frac{2r_0d}{2d-r_0(2s_c-p)}$ and $\frac{2r_0d}{2(1+p)(d-r_0s_c)+r_0p}$ are dual exponents, the dispersive estimate \eqref{dispersive-p} implies
$$
\bigl\|e^{i(t-s)\Delta}F(s)\bigr\|_{L_x^{\frac{2r_0d}{2d-r_0(2s_c-p)}}}
\lesssim |t-s|^{-\frac{p(d-r_0s_c)}{2r_0}} \|F(s)\|_{L_x^{\frac{2r_0d}{2(1+p)(d-r_0s_c)+r_0p}}}.
$$
Using the Hardy-Littlewood-Sobolev inequality and the fact that $(q_0,r_0)$ is a Schr\"odinger admissible pair, we obtain
$$
\Bigl\|\int_{t_0}^t e^{i(t-s)\Delta}F(s)ds\Bigr\|_{L_t^{q_0}L_x^{\frac{2r_0d}{2d-r_0(2s_c-p)}}(I\times\R^d)}
\lesssim \|F\|_{L_t^{\frac{q_0}{1+p}} L_x^{\frac{2r_0d}{2(1+p)(d-r_0s_c)+r_0p}}(I\times\R^d)}.
$$
As the differentiation operator $|\nabla|^{p/2}$ commutes with the free evolution, we recover the claim.
\end{proof}

Next we establish some connections between the spaces defined in \eqref{spaces} and the usual Strichartz spaces.

\begin{lemma}[Interpolations]\label{LT:L:interpolations}
For any compact time interval $I$,
\begin{align}
\|u\|_{X^0(I)}
&\lesssim \|u\|_{X(I)}\lesssim \bigl\||\nabla|^{s_c} u\bigr\|_{S^0(I)} \label{LT:X emb}\\
\|u\|_{X(I)}
&\lesssim \|u\|_{L_{t,x}^{\frac{p(d+2)}2}(I\times\R^d)}^{\theta_1} \bigl\||\nabla|^{s_c} u\bigr\|_{S^0(I)}^{1-\theta_1} \label{LT:inter1}\\
\|u\|_{L_{t,x}^{\frac {p(d+2)}2}(I\times\R^d)}
&\lesssim\|u\|_{X(I)}^{\theta_2} \bigl\||\nabla|^{s_c} u\bigr\|_{S^0(I)}^{1-\theta_2}, \label{LT:inter2}
\end{align}
where $0<\theta_1,\theta_2< 1$ depend on $d,p$.
\end{lemma}

\begin{proof}
A simple application of Sobolev embedding yields \eqref{LT:X emb}.

Using interpolation, we obtain
\begin{align*}
\|u\|_{X(I)}
&\lesssim \|u\|_{X^0(I)}^{1-\frac{p}{2s_c}} \bigl\||\nabla|^{s_c} u\bigr\|_{L_t^{q_0}L_x^{r_0}(I\times\R^d)}^{\frac{p}{2s_c}}
\lesssim \|u\|_{X^0(I)}^{1-\frac{p}{2s_c}} \bigl\||\nabla|^{s_c} u\bigr\|_{S^0(I)}^{\frac{p}{2s_c}}.
\end{align*}
On the other hand, as $\tfrac{p(d+2)}2< q_0<\infty$, interpolation followed by Sobolev embedding yields
$$
\|u\|_{X^0(I)}
\lesssim \|u\|_{L_{t,x}^{\frac {p(d+2)}2}(I\times\R^d)}^{\frac{p(d+2)}{2q_0}}
    \|u\|_{L_t^\infty L_x^{\frac {pd}2}(I\times\R^d)}^{1-\frac{p(d+2)}{2q_0}}
\lesssim \|u\|_{L_{t,x}^{\frac {p(d+2)}2}(I\times\R^d)}^{\frac{p(d+2)}{2q_0}} \bigl\||\nabla|^{s_c} u\bigr\|_{S^0(I)}^{1-\frac{p(d+2)}{2q_0}}.
$$
Putting everything together, we derive \eqref{LT:inter1}.

We now turn to \eqref{LT:inter2}; using interpolation once again, we obtain
$$
\|u\|_{L_{t,x}^{\frac {p(d+2)}2}(I\times\R^d)}
\lesssim \|u\|_{X^0(I)}^{\frac{q_0(pd-4)}{p[q_0d-2(d+2)]}}
    \|u\|_{L_t^{\frac{2(d+2)}d} L_x^{\frac{pd(d+2)}{2(d+2)-pd}} (I\times\R^d)}^{1-\frac{q_0(pd-4)}{p[q_0d-2(d+2)]}}
$$
and the claim follows from \eqref{LT:X emb} and Sobolev embedding.
\end{proof}

Finally, we derive estimates that will help us control the nonlinearity.  The main tools we use in deriving these estimates
are the fractional chain rules, Lemmas~\ref{L:Fract chain rule} and~\ref{L:FDFP}.

\begin{lemma}[Nonlinear estimates]\label{LT:L:nonlinear estimate}
Let $d\geq 5$ and assume the critical regularity $s_c$ satisfies \eqref{sc}.  Let $I$ a compact time interval.  Then,
\begin{align}
\|F(u)\|_{Y(I)}\lesssim \|u\|_{X(I)}^{p+1},\label{LT:whole nonlin}
\end{align}
\begin{align}
\|F_z(u+v)&w\|_{Y(I)}+\|F_{\bar z}  (u +v)\bar w \|_{Y(I)} \label{LT:nonlin est} \notag\\
&\lesssim \Bigl( \|u\|_{X(I)}^{\frac{p(s_c-1)}{s_c}}\bigl\||\nabla|^{s_c} u\bigr\|_{S^0(I)}^{\frac{p}{s_c}}
         + \|v\|_{X(I)}^{\frac{p(s_c-1)}{s_c}}\bigl\||\nabla|^{s_c} v\bigr\|_{S^0(I)}^{\frac{p}{s_c}}\Bigr) \|w\|_{X(I)},
\end{align}
and
\begin{align}\label{LT:nonlin sc}
&\bigl\||\nabla|^{s_c}[ F(u+v)- F(u)]\bigr\|_{N^0(I)}\\
&\lesssim \bigl\||\nabla|^{s_c} v\bigr\|_{S^0(I)} \Bigl[ \|v\|_{X^0(I)}^p + \|u\|_{X^0(I)}^{p-1+\frac1{s_c}}\|v\|_{X^0(I)}^{1-\frac1{s_c}}
    +\bigl( \|u\|_{X^0(I)}^{p-1+\frac1{s_c}} + \|v\|_{X^0(I)}^{p-1+\frac1{s_c}}\bigr)\bigl\||\nabla|^{s_c} u\bigr\|_{S^0(I)}^{1-\frac1{s_c}}\Bigr]\notag\\
&\quad +\bigl\||\nabla|^{s_c} u\bigr\|_{S^0(I)}\Bigl(\|v\|_{X^0(I)}^p + \|u\|_{X^0(I)}^{\beta}\|v\|_{X^0(I)}^{p-\beta}\Bigr)  \notag\\
&\quad + \bigl\||\nabla|^{s_c} u\bigr\|_{S^0(I)}^{\frac1{s_c}}\bigl\||\nabla|^{s_c} v\bigr\|_{S^0(I)}^{1-\frac1{s_c}}\|u\|_{X^0(I)}^{1-\frac1{s_c}}\|v\|_{X^0(I)}^{p-1+\frac1{s_c}}\notag
\end{align}
for some $0<\beta<p$.
\end{lemma}

\begin{proof}
Throughout the proof, all spacetime norms are on $I\times\R^d$.

Applying Lemma \ref{L:Fract chain rule} followed \eqref{LT:X emb}, we find
\begin{align*}
\|F(u)\|_{Y(I)}\lesssim \|u\|_{X^0(I)}^p \|u\|_{X(I)}\lesssim \|u\|_{X(I)}^{1+p}.
\end{align*}
This establishes \eqref{LT:whole nonlin}.

We now turn to \eqref{LT:nonlin est}; we only treat the first term on the left-hand side, as the second term can
be handled similarly.  By Lemma~\ref{L:product rule} followed by \eqref{LT:X emb},
\begin{align*}
\|F_z(u+ v)w\|_{Y(I)}
&\lesssim \| F_z(u+v)\|_{L_t^{\frac{q_0}p}L_x^{\frac{r_0d}{p(d-r_0s_c)}}} \|w\|_{X(I)}\\
&\quad + \bigl\| |\nabla|^{\frac p2} F_z(u+v) \bigr\|_{L_t^{\frac{q_0}p} L_x^{\frac{2r_0d}{p(2d-2r_0s_c+r_0)}}}\|w\|_{X^0(I)}\\
&\lesssim \Bigl(\|u+v\|_{X^0(I)}^p + \bigl\| |\nabla|^{\frac p2} F_z(u+v) \bigr\|_{L_t^{\frac{q_0}p} L_x^{\frac{2r_0d}{p(2d-2r_0s_c+r_0)}}}\Bigr)
        \|w\|_{X(I)}.
\end{align*}
Thus, the claim will follow from \eqref{LT:X emb}, once we establish
\begin{align}\label{LT:goal}
\bigl\| |\nabla|^{\frac p2} F_z(u+v) &\bigr\|_{L_t^{\frac{q_0}p} L_x^{\frac{2r_0d}{p(2d-2r_0s_c+r_0)}}}\notag\\
&\lesssim \|u\|_{X(I)}^{\frac{p(s_c-1)}{s_c}}\bigl\||\nabla|^{s_c} u\bigr\|_{S^0(I)}^{\frac{p}{s_c}}
         + \|v\|_{X(I)}^{\frac{p(s_c-1)}{s_c}}\bigl\||\nabla|^{s_c} v\bigr\|_{S^0(I)}^{\frac{p}{s_c}}.
\end{align}
For $p\geq 1$, this follows from Lemma~\ref{L:Fract chain rule} and \eqref{LT:X emb}:
\begin{align*}
\bigl\| |\nabla|^{\frac p2} F_z(u+v) \bigr\|_{L_t^{\frac{q_0}p} L_x^{\frac{2r_0d}{p(2d-2r_0s_c+r_0)}}}
&\lesssim \|u+v\|_{X^0(I)}^{p-1}\|u+v\|_{X(I)}
\lesssim \|u+v\|_{X(I)}^p.
\end{align*}
To derive \eqref{LT:goal} for $p<1$, we apply Lemma~\ref{L:FDFP} (with $s:=p/2$ and $1/2<\sigma<1$)
followed by H\"older's inequality in the time variable, Sobolev embedding, and interpolation:
\begin{align*}
\bigl\| |\nabla|^{\frac p2} F_z(u+v)  \bigr\|_{L_t^{\frac{q_0}p} L_x^{\frac{2r_0d}{p(2d-2r_0s_c+r_0)}}}
&\lesssim \|u+v\|_{X^0(I)}^{p-\frac{p}{2\sigma}}
    \bigl\| |\nabla|^\sigma (u +v) \bigr \|_{L_t^{q_0}L_x^{\frac {r_0d}{d-r_0(s_c-\sigma)}}}^{\frac{p}{2\sigma}}\\
&\lesssim \bigl\| |\nabla|^\sigma (u+v) \bigr \|_{L_t^{q_0}L_x^{\frac {r_0d}{d-r_0(s_c-\sigma)}}}^p\\
&\lesssim \|u\|_{X^0(I)}^{p- \frac{p\sigma}{s_c}}\bigl\| |\nabla|^{s_c} u\bigr\|_{S^0(I)}^{\frac{p\sigma}{s_c}}
         + \|v\|_{X^0(I)}^{p- \frac{p\sigma}{s_c}}\bigl\| |\nabla|^{s_c} v\bigr\|_{S^0(I)}^{\frac{p\sigma}{s_c}}.
\end{align*}
Invoking \eqref{LT:X emb}, this settles \eqref{LT:goal} and hence \eqref{LT:nonlin est}.

To prove \eqref{LT:nonlin sc}, we estimate
\begin{align}\label{sc deriv of non}
\bigl\||\nabla|^{s_c} [F(u+v)-F(u)]\bigr\|_{N^0(I)}
&\lesssim \bigl\||\nabla|^{s_c-1} \bigl[\nabla v \cdot F'(u+v)\bigr]\bigr\|_{N^0(I)}\notag\\
&\quad + \bigl\||\nabla|^{s_c-1} \bigl[\nabla u \cdot \bigl(F'(u+v)-F'(u)\bigr)\bigr]\bigr\|_{N^0(I)}.
\end{align}
To estimate the first term on the right-hand side of \eqref{sc deriv of non}, we use Lemmas~\ref{L:product rule} and \ref{L:FDFP}
together with H\"older's inequality and interpolation:
\begin{align*}
\bigl\||\nabla|^{s_c-1} & \bigl[\nabla v \cdot F'(u+v)\bigr]\bigr\|_{N^0(I)}\\
&\lesssim \bigl\||\nabla|^{s_c} v\bigr\|_{S^0(I)} \bigl( \|u\|_{X^0(I)}^p +  \|v\|_{X^0(I)}^p \bigr)\\
&\quad + \bigl\||\nabla|^{s_c} v\bigr\|_{S^0(I)}^{\frac1{s_c}} \|v\|_{X^0(I)}^{1-\frac1{s_c}} \|u+v\|_{X^0(I)}^{p-1+\frac1{s_c}}
    \bigl\||\nabla|^{s_c} (u+v) \bigr\|_{S^0(I)}^{1-\frac1{s_c}}\\
&\lesssim \bigl\||\nabla|^{s_c} v\bigr\|_{S^0(I)} \Bigl[ \|v\|_{X^0(I)}^p + \|u\|_{X^0(I)}^{p-1+\frac1{s_c}}\|v\|_{X^0(I)}^{1-\frac1{s_c}}\\
&\qquad\qquad\qquad\qquad\qquad\quad\,\, +\bigl( \|u\|_{X^0(I)}^{p-1+\frac1{s_c}} + \|v\|_{X^0(I)}^{p-1+\frac1{s_c}}\bigr)\bigl\||\nabla|^{s_c} u\bigr\|_{S^0(I)}^{1-\frac1{s_c}}\Bigr].
\end{align*}
To estimate the second term on the right-hand side of \eqref{sc deriv of non}, we use Lemma~\ref{L:frac deriv of diff}
together with H\"older's inequality, interpolation, and \eqref{LT:X emb}:
\begin{align*}
\bigl\||\nabla|^{s_c-1} &  \bigl[\nabla u\cdot \bigl(F'(u+v)-F'(u)\bigr)\bigr]\bigr\|_{N^0(I)}\\
&\lesssim \bigl\||\nabla|^{s_c} u\bigr\|_{S^0(I)}\|v\|_{X^0(I)}^p + \bigl\||\nabla|^{s_c} u\bigr\|_{S^0(I)}\|u\|_{X^0(I)}^{\frac{s_c-1}\sigma}\|v\|_{X^0(I)}^{p-\frac{s_c-1}\sigma}\\
&\quad + \bigl\||\nabla|^{s_c} u\bigr\|_{S^0(I)}^{\frac1{s_c}} \|u\|_{X^0(I)}^{1-\frac1{s_c}} \bigl\||\nabla|^{s_c}v\bigr\|_{S^0(I)}^{1-\frac1{s_c}}\|v\|_{X^0(I)}^{p-1+\frac1{s_c}},
\end{align*}
where $s_c-1<\sigma p<p$.  Denoting $\beta:= \frac{s_c-1}\sigma$ and collecting all the estimates above we derive \eqref{LT:nonlin sc}.
\end{proof}

We have now all the tools we need to attack Theorem~\ref{T:stab}.  We start with the following:

\begin{lemma}[Short-time perturbations]\label{L:short time stab}
Let $d\geq 5$ and assume the critical regularity $s_c$ satisfies \eqref{sc}.  Let $I$ be a compact time interval containing zero and let $\tilde u$
be an approximate solution to \eqref{nls} on $I\times\R^d$ in the sense that
$$
i\tilde u_t =-\Delta \tilde u + F(\tilde u) + e
$$
for some function $e$.  Assume that
$$
\|\tilde u\|_{L_t^{\infty}\dot H_x^{s_c}(I\times\R^d)}\leq E
$$
for some positive constant $E$.  Moreover, let $u_0\in \dot H^{s_c}_x$ and assume that
\begin{align}
\|\tilde u\|_{X(I)}&\le \delta \label{LT:small u in X}\\
\|u_0-\tilde u_0\|_{\dot H_x^{s_c}}&\leq \eps \label{LT:close id}\\
\bigl\||\nabla|^{s_c} e\bigr\|_{N^0(I)}&\le \eps \label{LT:small error}
\end{align}
for some small $0<\delta=\delta(E)$ and $0<\eps<\eps_0(E)$.  Then there exists a unique solution $u:I\times\R^d\to\C$
to \eqref{nls} with initial data $u_0$ at time $t=0$; it satisfies
\begin{align}
\|u-\tilde u\|_{X(I)}&\lesssim \eps \label{LT:diff in X}\\
\bigl\||\nabla|^{s_c} (u-\tilde u)\bigr\|_{S^0(I)}&\lesssim \eps^{c(d,p)} \label{LT:diff in S}\\
\bigl\||\nabla|^{s_c}u \bigr\|_{S^0(I)}&\lesssim E \label{LT:u in S}\\
\|F(u)-F(\tilde u)\|_{Y(I)}&\lesssim\eps \label{LT:diff nonlin in Y}\\
\bigl\||\nabla|^{s_c} \bigl[F(u)-F(\tilde u)\bigr]\bigr\|_{N^0(I)}&\lesssim \eps^{c(d,p)}, \label{LT:diff nonlin in N}
\end{align}
for some positive constant $c(d,p)$.
\end{lemma}

\begin{proof}
As explained at the beginning of the proof of Theorem~\ref{T:stab cubic}, we may assume that $u$ exists and merely show that it obeys
the estimates stated above.

We start by deriving some bounds on $\tilde u$ and $u$.  By Strichartz, Corollary~\ref{C:s deriv}, Lemma~\ref{LT:L:interpolations},
\eqref{LT:small u in X}, and \eqref{LT:small error},
\begin{align*}
\bigl\||\nabla|^{s_c} \tilde u\bigr\|_{S^0(I)}
&\lesssim \|\tilde u\|_{L_t^\infty \dot H^{s_c}_x(I\times\R^d)} + \bigl\||\nabla|^{s_c} F(\tilde u)\bigr\|_{N^0(I)}
    + \bigl\||\nabla|^{s_c} e\bigr\|_{N^0(I)} \\
&\lesssim E + \|\tilde u\|_{L_{t,x}^{\frac{p(d+2)}2}(I\times\R^d)}^p \bigl\||\nabla|^{s_c} \tilde u\bigr\|_{S^0(I)} +\eps\\
&\lesssim E + \delta^{p\theta_2} \bigl\||\nabla|^{s_c} \tilde u\bigr\|_{S^0(I)}^{1+p(1-\theta_2)} +\eps,
\end{align*}
where $\theta_2$ is as in Lemma~\ref{LT:L:interpolations}.  Choosing $\delta$ small depending on $d,p,E$ and $\eps_0$ sufficiently small
depending on $E$, we obtain
\begin{equation}\label{LT:utilde in S}
\bigl\||\nabla|^{s_c} \tilde u\bigr\|_{S^0(I)}\lesssim E.
\end{equation}
Moreover, by Lemma~\ref{LT:L:exotic strichartz}, Lemma~\ref{LT:L:nonlinear estimate}, \eqref{LT:small u in X}, and \eqref{LT:small error},
\begin{align*}
\bigl\| e^{it\Delta} \tilde u_0 \bigr\|_{X(I)}
&\lesssim \|\tilde u\|_{X(I)} + \|F(\tilde u)\|_{Y(I)} + \bigl\||\nabla|^{s_c} e\bigr\|_{N^0(I)}
\lesssim \delta + \delta^{\frac {d+2}{d-2}} + \eps
\lesssim \delta,
\end{align*}
provided $\delta$ and $\eps_0$ are chosen sufficiently small.  Combining this with the triangle inequality, \eqref{LT:X emb}, the Strichartz inequality,
and \eqref{LT:close id}, we obtain
\begin{align*}
\bigl\|e^{it\Delta} u_0 \bigr\|_{X(I)}
\lesssim \bigl\| e^{it\Delta} \tilde u_0 \bigr\|_{X(I)}+ \|u_0-\tilde u_0\|_{\dot H^{s_c}_x}
\lesssim \delta + \eps
\lesssim \delta.
\end{align*}
Thus, another application of Lemma~\ref{LT:L:exotic strichartz} combined with Lemma~\ref{LT:L:nonlinear estimate} gives
$$
\|u\|_{X(I)}
\lesssim \bigl\|e^{it\Delta} u_0\bigr\|_{X(I)} + \|F(u)\|_{Y(I)}
\lesssim \delta + \|u\|_{X(I)}^{\frac{d+2}{d-2}}.
$$
Choosing $\delta$ sufficiently small, the usual bootstrap argument yields
\begin{align}\label{LT:u in X}
\|u\|_{X(I)} \lesssim \delta.
\end{align}

Next we derive the claimed bounds on $w:=u-\tilde u$.  Note that $w$ is a solution to
\begin{equation*}
\begin{cases}
iw_t=-\Delta w + F(\tilde u + w)-F(\tilde u)- e\\
w(t_0)=u_0-\tilde u_0.
\end{cases}
\end{equation*}
Using Lemma~\ref{LT:L:exotic strichartz} together with Lemma~\ref{LT:L:interpolations}, the Strichartz inequality,
\eqref{LT:close id}, and \eqref{LT:small error}, we see that
\begin{align*}
\|w\|_{X(I)}
&\lesssim \|u_0-\tilde u_0\|_{\dot H_x^{s_c}}+\bigl\||\nabla|^{s_c} e\bigr\|_{N^0(I)}
    +\|F(u)-F(\tilde u)\|_{Y(I)}
\lesssim \eps + \|F(u)-F(\tilde u)\|_{Y(I)}.
\end{align*}
To estimate the difference of the nonlinearities, we use Lemma~\ref{LT:L:nonlinear estimate}, \eqref{LT:small u in X}, \eqref{LT:utilde in S}:
\begin{align}\label{LT:diff F(u) in Y}
\|F(u)-F(\tilde u)\|_{Y(I)}
&\lesssim \Bigl[ \|\tilde u\|_{X(I)}^{\frac{p(s_c-1)}{s_c}}\bigl\||\nabla|^{s_c} \tilde u\bigr\|_{S^0(I)}^{\frac{p}{s_c}}
         + \|w\|_{X(I)}^{\frac{p(s_c-1)}{s_c}}\bigl\||\nabla|^{s_c} w\bigr\|_{S^0(I)}^{\frac{p}{s_c}} \Bigr] \|w\|_{X(I)}\notag\\
&\lesssim \delta^{\frac{p(s_c-1)}{s_c}} E^{\frac{p}{s_c}} \|w\|_{X(I)}
        + \bigl\||\nabla|^{s_c} w\bigr\|_{S^0(I)}^{\frac{p}{s_c}} \|w\|_{X(I)}^{1+\frac{p(s_c-1)}{s_c}}.
\end{align}
Thus, choosing $\delta$ sufficiently small depending only on $E$, we obtain
\begin{align}\label{LT:w in X}
\|w\|_{X(I)} \lesssim \eps + \bigl\||\nabla|^{s_c} w\bigr\|_{S^0(I)}^{\frac{p}{s_c}} \|w\|_{X(I)}^{1+\frac{p(s_c-1)}{s_c}}.
\end{align}

On the other hand, by the Strichartz inequality and the hypotheses,
\begin{align}\label{LT:w in S est}
\bigl\||\nabla|^{s_c} w\bigr\|_{S^0(I)}
&\lesssim \|u_0-\tilde u_0\|_{\dot H_x^{s_c}}  + \bigl\||\nabla|^{s_c} e\bigr\|_{ N^0(I)} + \bigl\| |\nabla|^{s_c} \bigl[ F(u)-F(\tilde u) \bigr] \bigr\|_{N^0(I)}\notag\\
&\lesssim \eps + \bigl\| |\nabla|^{s_c} \bigl[ F(u)-F(\tilde u) \bigr] \bigr\|_{N^0(I)}.
\end{align}
To estimate the difference of the nonlinearities, we use \eqref{LT:nonlin sc} together with Lemma~\ref{LT:L:interpolations},
\eqref{LT:small u in X}, \eqref{LT:utilde in S}, and \eqref{LT:u in X},
\begin{align}\label{LT:diff F(u) in N}
\bigl\| |\nabla|^{s_c}\bigl[ F(u)-F(\tilde u) \bigr] \bigr\|_{N^0(I)}
&\lesssim \bigl\| |\nabla|^{s_c} w\bigr\|_{S^0(I)}\Bigl(\delta +\delta^{p-1+\frac 1{s_c}} E^{1-\frac1{s_c}} \Bigr)\\
&\quad +  E \delta^{\beta}\|w\|_{X(I)}^{p-\beta} + \delta^{1-\frac1{s_c}}E^{\frac1{s_c}}\bigl\| |\nabla|^{s_c} w\bigr\|_{S^0(I)}^{1-\frac1{s_c}}\|w\|_{X(I)}^{p-1+\frac1{s_c}}\notag
\end{align}
for some $0<\beta<p$.  Thus, choosing $\delta$ small depending only on $E$, \eqref{LT:w in S est} implies
\begin{align}\label{LT:w in S}
\bigl\| |\nabla|^{s_c} w\bigr\|_{S^0(I)}
\lesssim \eps + \|w\|_{X(I)}^{p-\beta} + \bigl\| |\nabla|^{s_c} w\bigr\|_{S^0(I)}^{1-\frac1{s_c}} |w\|_{X(I)}^{p-1+\frac1{s_c}}.
\end{align}

Combining \eqref{LT:w in X} with \eqref{LT:w in S}, the usual bootstrap argument yields \eqref{LT:diff in X} and \eqref{LT:diff in S},
provided  $\eps_0$ is chosen sufficiently small depending on $E$.  By the triangle inequality, \eqref{LT:diff in S}
and \eqref{LT:utilde in S} imply \eqref{LT:u in S}.

Claims \eqref{LT:diff nonlin in Y} and \eqref{LT:diff nonlin in N} follow from \eqref{LT:diff F(u) in Y} and \eqref{LT:diff F(u) in N}
combined with \eqref{LT:diff in X} and \eqref{LT:diff in S}, provided we take $\delta$ and $\eps_0$ sufficiently small depending on $E$.
\end{proof}

We are finally in a position to prove the stability result.

\begin{proof}[Proof of Theorem~\ref{T:stab}]
Our first goal is to show
\begin{equation}\label{LT:tilde u}
\bigl\| |\nabla|^{s_c} \tilde u\bigr\|_{S^0(I)}\le C(E,L).
\end{equation}
Indeed, by \eqref{LT: finite s-t} we may divide $I$ into $J_0=J_0(L, \eta)$ subintervals $I_j=[t_j,t_{j+1}]$
such that on each spacetime slab $I_j\times\R^d$
$$
\|\tilde u\|_{L_{t,x}^{\frac {p(d+2)}2}(I_j\times\R^d)}\le \eta
$$
for a small constant $\eta>0$ to be chosen in a moment.  By the Strichartz inequality combined with Corollary~\ref{C:s deriv}, \eqref{LT:finite energy}
and \eqref{LT:error small},
\begin{align*}
\bigl\| |\nabla|^{s_c} \tilde u\bigr\|_{S^0(I_j)}
&\lesssim \|\tilde u(t_j)\|_{\dot H_x^{s_c}} + \bigl\| |\nabla|^{s_c} e \bigr\|_{ N^0(I_j)} + \bigl\| |\nabla|^{s_c} F(\tilde u)\bigr\|_{N^0(I_j)}\\
&\lesssim E + \eps + \|\tilde u\|_{L_{t,x}^{\frac {p(d+2)}2}(I_j\times\R^d)}^p \bigl\| |\nabla|^{s_c} \tilde u\bigr\|_{S^0(I_j)}\\
&\lesssim E + \eps + \eta^p \bigl\| |\nabla|^{s_c} \tilde u\bigr\|_{S^0(I_j)}.
\end{align*}
Thus, choosing $\eta>0$ small depending on $d,p$ and $\eps_1$ sufficiently small depending on $E$, we obtain
$$
\bigl\| |\nabla|^{s_c} \tilde u\bigr\|_{S^0(I_j)}\lesssim E.
$$
Summing this over all subintervals $I_j$, we derive \eqref{LT:tilde u}.

Using Lemma~\ref{LT:L:interpolations} together with \eqref{LT:tilde u}, we obtain
\begin{align}
\|\tilde u\|_{X(I)}&\le C(E,L) \label{LT:utilde bdd in X}
\end{align}
By \eqref{LT:utilde bdd in X}, we may divide $I$ into $J_1=J_1(E,L)$ subintervals $I_j=[t_j,t_{j+1}]$ such that on
each spacetime slab $I_j\times\R^d$
$$
\|\tilde u\|_{X(I_j)}\le \delta
$$
for some small $\delta=\delta(E)>0$ as in Lemma~\ref{L:short time stab}.

Thus, choosing $\eps_1$ sufficiently small (depending on $J_1$ and $E$), we may apply Lemma~\ref{L:short time stab}
to obtain for each $0\leq j< J_1$ and all $0<\eps<\eps_1$,
\begin{equation}\label{bounds on j}
\begin{aligned}
\|u-\tilde u\|_{X(I_j)} &\leq C(j)\eps^{c(d,p)^j}\\
\bigl\| |\nabla|^{s_c}( u-\tilde u) \bigr\|_{S^0(I_j)}&\leq C(j) \eps^{c(d,p)^{j+1}}\\
\bigl\| |\nabla|^{s_c} u \bigr\|_{\dot S^1(I_j)}&\leq C(j) E\\
\|F(u)-F(\tilde u)]\|_{Y(I_j)}&\leq C(j)\eps^{c(d,p)^j}\\
\bigl\| |\nabla|^{s_c}\bigl[F(u)-F(\tilde u)\bigr] \bigr\|_{N^0(I_j)}&\leq C(j)\eps^{c(d,p)^{j+1}},
\end{aligned}
\end{equation}
provided we can show
\begin{align}\label{LT:left}
\|u(t_j)-\tilde u(t_j)\|_{\dot H^{s_c}}&\lesssim \eps^{c(d,p)^j}\leq \eps_0
\end{align}
for each $0\leq j<J_1$, where $\eps_0$ is as in Lemma~\ref{L:short time stab}.  By the Strichartz inequality and the inductive hypothesis,
\begin{align*}
\|u(t_{j})- \tilde u(t_{j})\|_{\dot{H}^{s_c}_x}
&\lesssim \|u_0-\tilde u_0\|_{\dot{H}^{s_c}_x} + \bigl\| |\nabla|^{s_c} e \bigr\|_{N^0([0, t_{j}])}
    + \bigl\| |\nabla|^{s_c} \bigl[ F(u)-F(\tilde u) \bigr]\bigr\|_{N^0([0, t_j])}\\
&\lesssim \eps +\sum_{k=0}^{j-1}C(k) \eps^{c(d,p)^{k+1}}.
\end{align*}
Taking $\eps_1$ sufficiently small compared to $\eps_0$, we see that \eqref{LT:left} is satisfied.

Summing the bounds in \eqref{bounds on j} over all subintervals $I_j$ and using Lemma~\ref{LT:L:interpolations}, we derive \eqref{LT:close in s-t}
through \eqref{LT:u in Sc}.  This completes the proof of the theorem.
\end{proof}

%
%
%
%

\section{Reduction to almost periodic solutions}\label{S:Reduct}

The goal of this section is to prove Theorem~\ref{T:reduct}.  In order to achieve this, we repeat the argument presented in \cite{Berbec}.
Since the procedure is by now standard, we content ourselves with identifying the main steps and indicating, whenever necessary, the changes
that appear with respect to the presentation in \cite{Berbec}.  See also \cite{kenig-merle, kenig-merle:1/2, kenig-merle:wave sup, Notes, tvz:cc}
for similar arguments in other contexts.

We start by presenting the setup.  Throughout this section, we fix a dimension $d\geq 5$ and assume that either the critical regularity is
$s_c=\frac{d-2}2$ or it satisfies \eqref{sc}.  For any $0\leq E_0<\infty$, we define
$$
L(E_0):=\sup \{S_I(u):\, u:\ird\to \C \text{ such that } \sup_{t\in I}\bigl\| |\nabla|^{s_c} u(t)\bigr\|_2^2 \leq  E_0\},
$$
where the supremum  is taken over all solutions $u:\ird\to \C$ to \eqref{nls} obeying $\bigl\| |\nabla|^{s_c} u(t)\bigr\|_2^2\leq E_0$.
Thus, $L:\bigl[0,\infty) \to [0, \infty]$ is a non-decreasing function.  Moreover, from Theorem~\ref{T:local},
\begin{align*}
L(E_0)\lesssim_d E_0^{\frac{p(d+2)}4} \quad \text{for} \quad  E_0\leq \eta_0,
\end{align*}
where $\eta_0=\eta_0(d,p)$ is the threshold from the small data theory.

From Theorems~\ref{T:stab cubic} and \ref{T:stab}, we see that $L$ is continuous in both settings considered here.  Therefore, there must exist
a unique \emph{critical} $E_c\in(0,\infty]$ such that $L(E_0)<\infty$ for $E_0<E_c$ and $L(E_0)=\infty$ for $E_0\geq E_c$.  In particular, if
$u:\ird\to \C$ is a maximal-lifespan solution to \eqref{nls} such that $\sup_{t\in I}\bigl\| |\nabla|^{s_c} u(t)\bigr\|_2^2 < E_c$, then $u$
is global and moreover,
$$
S_\R(u)\leq L\bigl(\|u\|_{L_t^\infty \dot H^{s_c}_x}^2\bigr).
$$
Failure of Theorems~\ref{T:main cubic} or \ref{T:main} is equivalent to $0 < E_c < \infty$.

Following the presentation in \cite{Berbec, tvz:cc}, the main step in proving Theorem~\ref{T:reduct} is to prove a Palais--Smale condition
modulo the symmetries of the equation.  With the Palais--Smale condition in place, the proof of Theorem~\ref{T:reduct} is standard;
see, for example, \cite{kenig-merle, kenig-merle:1/2, Berbec, tvz:cc}.

\begin{proposition}[Palais-Smale condition modulo symmetries]\label{P:palais-smale}
Assume we are either in the setting of Theorem~\ref{T:main cubic} or Theorem~\ref{T:main}.  Let $u_n:I_n\times\R^d\mapsto \C$ be a sequence of solutions
to \eqref{nls} such that
\begin{align}\label{max ke}
\limsup_{n\to \infty} \sup_{t\in I_n}\bigl\| |\nabla|^{s_c} u_n(t)\bigr\|_2^2 =E_c
\end{align}
and let $t_n\in I_n$ be a sequence of times such that
\begin{equation*}
\lim_{n\to \infty} S_{\ge t_n}(u_n) = \lim_{n\to \infty} S_{\le t_n}(u_n) = \infty.
\end{equation*}
Then the sequence $u_n(t_n)$ has a subsequence which converges in $\dot H^{s_c}_x(\R^d)$ modulo symmetries.
\end{proposition}

The proof of this proposition follows the same recipe as that used to prove the analogous statement in \cite{Berbec}.
The main ingredients are the linear profile decomposition from Lemma~\ref{L:cc} and the stability results Theorems~\ref{T:stab cubic}
and~\ref{T:stab}.  The only new difficulty appears when one endeavors to prove decoupling of the nonlinear profiles.
In the energy-critical setting, one uses the pointwise inequality
$$
\Bigl|\nabla \Bigl(\sum_{j=1}^J F(f_j) - F\bigl(\sum_{j=1}^J f_j\bigr) \Bigr) \Bigr| \lesssim_J \sum_{j\neq j'} |\nabla f_j|
|f_{j'}|^{\frac{4}{d-2}},
$$
which does not carry over to the case of a non-integer number of derivatives.  Let us first discuss the case $0<p\leq 1$; in this scenario,
we use \eqref{E:pointwise} to obtain the following substitute for the inequality above:
\begin{align*}
\mathcal{D}_{s_c-1}&\biggl( \nabla \Bigl[\sum_{j=1}^J F(f_j) - F\bigl(\sum_{j=1}^J f_j\bigr) \Bigr] \biggr)\\
&\lesssim \mathcal{D}_{s_c-1}\biggl( \sum_{j=1}^J\nabla f_j \Bigl[ F'(f_j) - F'\bigl(\sum_{j'=1}^J f_{j'}\bigr) \Bigr] \biggr)\\
&\lesssim_J \sum_{j\neq j'} \biggl\{\mathcal{D}_{s_c-1}\bigl(\nabla f_j\bigr) |f_{j'}|^p
    + \biggl[\sum_{l=1}^J \mathcal D_\sigma(f_l)\biggr]^{\frac s\sigma} M\bigl(|\nabla f_j|^{\frac1{1-p}} \bigr)^{1-p}
        M\bigl(f_{j'}\bigr)^{p-\frac s\sigma}\biggr\}
\end{align*}
for some $s_c-1<\sigma p<p$.  We remind the reader that $M$ denotes the Hardy--Littlewood maximal function, which commutes with the symmetries
of the equation.  The operator $\mathcal D_s$ is defined in \eqref{Ds def} and it behaves like $|\nabla |^s$ under symmetries.

In the case discussed above, one has both a small-power non-polynomial nonlinearity and a non-integer number of derivatives $1<s_c<2$,
which makes it the most awkward of the scenarios we need to consider.  The remaining cases of Theorems~\ref{T:main cubic} and~\ref{T:main},
can be handled using various permutations of the techniques discussed above or some alternatives.  In particular, we draw the reader's attention
to \cite{kenig-merle:1/2} which considers the cubic nonlinearity with $s_c=\frac12$.

%
%
%
%

\section{The finite-time blowup solution}\label{S:self-similar}

In this section we preclude scenario I described in Theorem~\ref{T:enemies}.  We start by considering finite-time blowup solutions in the setting
of Theorem~\ref{T:main}.

\begin{theorem}[Absence of finite-time blowup solutions]
Let $d\geq 5$ and assume the critical regularity $s_c$ satisfies \eqref{sc}.  Then there are no finite-time blowup solutions to \eqref{nls}
in the sense of Theorem~\ref{T:enemies}.
\end{theorem}

\begin{proof}
We argue by contradiction.  Assume that there exists a solution $u:I\times\R^d\to \C$ that is a finite-time blowup solution in the sense
of Theorem~\ref{T:enemies}.  Assume also, without loss of generality, that the solution $u$ blows up in finite time in the future, that is,
$T:=\sup I<\infty$.

By hypothesis and Sobolev embedding,
\begin{align*}
\|u \|_{L_t^\infty L_x^{\frac{dp}2}(I\times\R^d)}
\lesssim \bigl\||\nabla|^{s_c-1}u \bigr\|_{L_t^\infty L_x^{\frac{2d}{d-2}}(I\times\R^d)}
\lesssim \|u\|_{L_t^\infty \dot H^{s_c}_x(I\times\R^d)} \lesssim_u 1.
\end{align*}
Thus, using the Duhamel formula \eqref{Duhamel} into the future together with the Strichartz and H\"older inequalities, as well as the fractional
chain rule, we obtain
\begin{align*}
\bigl\||\nabla|^{s_c-1}u(t)\bigr\|_2
&\leq \Bigl\| \int_t^T e^{i(t-s)\Delta} |\nabla|^{s_c-1} F(u(s))\, ds\Bigr\|_2 \notag\\
&\lesssim \bigl\| |\nabla|^{s_c-1} F(u) \bigr\|_{L_t^2L_x^{\frac{2d}{d+2}}([t,T)\times\R^d)} \notag\\
&\lesssim (T-t)^{\frac12}\bigl\||\nabla|^{s_c-1}u \bigr\|_{L_t^\infty L_x^{\frac{2d}{d-2}}([t,T)\times\R^d)}
    \|u \|_{L_t^\infty L_x^{\frac{dp}2}([t,T)\times\R^d)}^p \notag\\
&\lesssim_u (T-t)^{\frac12}.
\end{align*}
Interpolating with $u\in L_t^\infty \dot H^{s_c}_x$ (recalling that $1<s_c<2$ by \eqref{sc}) we derive that the energy $E(u)\to 0$ as $t\to T$.
Invoking the conservation of energy, we deduce that $u\equiv 0$.  This contradicts the fact that $u$ is a blowup solution.
\end{proof}

We consider next finite-time blowup solutions in the setting of Theorem~\ref{T:main cubic}.

\begin{theorem}[Absence of finite-time blowup solutions -- the cubic]
Let $d\geq 5$ and assume $p=2$.  Then there are no finite-time blowup solutions to \eqref{nls} in the sense of Theorem~\ref{T:enemies}.
\end{theorem}

\begin{proof}
Again, we argue by contradiction.  Let $u:I\times\R^d\to \C$ be a finite-time blowup solution in the sense of Theorem~\ref{T:enemies} and
assume that $T:=\sup I<\infty$.

By Sobolev embedding and the hypothesis,
\begin{align*}
\|u \|_{L_t^\infty L_x^d(I\times\R^d)}
\lesssim \bigl\||\nabla|^{\frac{d-4}2}u \bigr\|_{L_t^\infty L_x^{\frac{2d}{d-2}}(I\times\R^d)}
\lesssim \|u\|_{L_t^\infty \dot H^{\frac{d-2}2}_x(I\times\R^d)} \lesssim_u 1.
\end{align*}
Thus, using the Duhamel formula \eqref{Duhamel} into the future together with the Strichartz and H\"older inequalities, as well as the fractional
chain rule, we obtain
\begin{align}\label{one less deriv}
\bigl\||\nabla|^{\frac{d-4}2}u(t)\bigr\|_2
&\leq \Bigl\| \int_t^T e^{i(t-s)\Delta} |\nabla|^{\frac{d-4}2} F(u(s))\, ds\Bigr\|_2 \notag\\
&\lesssim \bigl\| |\nabla|^{\frac{d-4}2} F(u) \bigr\|_{L_t^2L_x^{\frac{2d}{d+2}}([t,T)\times\R^d)} \notag\\
&\lesssim (T-t)^{\frac12}\bigl\||\nabla|^{\frac{d-4}2}u \bigr\|_{L_t^\infty L_x^{\frac{2d}{d-2}}([t,T)\times\R^d)}
    \|u \|_{L_t^\infty L_x^d([t,T)\times\R^d)}^2 \notag\\
&\lesssim_u (T-t)^{\frac12}.
\end{align}
In particular, $u(t)\in \dot H^{\frac{d-4}2}_x$.

In the case when $d=5$, one can interpolate between \eqref{one less deriv} and the hypothesis
$u\in L_t^\infty \dot H^{\frac{d-2}2}_x(I\times\R^d)$ to derive that $\|\nabla u(t)\|_2\to 0$ as $t\to T$, and hence, by Sobolev embedding,
the energy $E(u(t))\to 0$ as $t\to T$.  Using the conservation of energy, we deduce that $u\equiv 0$.  This contradicts the fact that for a
finite-time blowup solution, $S_I(u)=\infty$.

To handle higher dimensions, we iterate the computations in \eqref{one less deriv} with one less derivative to deduce that
$$
\bigl\||\nabla|^{\frac{d-6}2}u(t)\bigr\|_2\lesssim_u (T-t).
$$
For $d=6$ this immediately implies that $u$ must have zero mass, while for $d=7$, the argument used to handle dimension $d=5$ implies that $u$
must have zero energy.  In both cases, we derive a contradiction to the fact that $u$ is a blowup solution.

To derive a contradiction for dimensions $d\geq 8$, we iterate the argument presented above.

This finishes the proof of the theorem.
\end{proof}

%
%
%
%

\section{Negative regularity}\label{S:neg}

In this section we prove that in scenarios II and III described in Theorem~\ref{T:enemies}, the solution $u$ admits negative regularity;
more precisely, it lies in $L_t^\infty \dot H^{-\eps}_x$ for some $\eps>0$.  In particular, this shows that the solution decays sufficiently
rapidly (in space) to belong to $L_t^\infty L_x^2$.  We first consider the setting of Theorem~\ref{T:main}.
At the end of this section we explain the changes needed to prove negative regularity in the setting of Theorem~\ref{T:main cubic}.

\begin{theorem}[Negative regularity for scenarios II and III]\label{T:-reg}
Let $d\geq 5$ and assume the critical regularity $s_c$ obeys \eqref{sc}.  Let $u$ be a global solution to \eqref{nls}
that is almost periodic modulo symmetries.  Suppose also that
\begin{align}\label{Hs bounded}
\sup_{t\in I} \bigl\| |\nabla|^{s_c} u(t)\bigr\|_{L_x^2} <\infty
\end{align}
and
\begin{align}\label{inf bounded}
\inf_{t\in I} N(t)\geq 1.
\end{align}
Then $u\in L_t^\infty \dot H_x^{-\eps}$ for some $\eps=\eps(d)>0$.  In particular, $u\in L^\infty_t L^2_x$.
\end{theorem}

To prove Theorem~\ref{T:-reg}, we employ the argument used in \cite{Berbec} to treat the energy-critical case, $s_c=1$.
We achieve our goal in two steps:  First, we `break' scaling in a Lebesque space; more precisely, we
prove that our solution lives in $L^\infty_t L_x^q$ for some $2<q<\tfrac{dp}2$.  Next, we use a double Duhamel trick to
upgrade this to $u\in L_t^\infty \dot H_x^{s_c-s_0}$ for some $s_0=s_0(d,p,q)>0$.  Iterating the second step finitely many times, we
arrive at Theorem~\ref{T:-reg}.

Let $u$ be a solution to \eqref{nls} that obeys the hypotheses of Theorem~\ref{T:-reg}.  Let $\eta>0$ be a small constant to be
chosen later.  Then by Remark~\ref{R:c small} combined with \eqref{inf bounded}, there exists $N_0=N_0(\eta)$ such that
\begin{align}\label{Hs small}
\bigl\| |\nabla|^{s_c} u_{\leq N_0}\bigr\|_{L_t^\infty L_x^2}\leq \eta.
\end{align}

We turn now to our first step, that is, breaking scaling in a Lebesgue space.  To this end, we define
\begin{equation*}
A(N):=
\begin{cases}
N^{\frac{d-3}2-\frac2p} \sup_{t\in \R} \|u_N(t)\|_{L_x^{\frac{2d}{d-3}}} \quad & \text{for } d=5,6\\
N^{\frac{d-2}2-\frac2p-\frac p2} \sup_{t\in \R} \|u_N(t)\|_{L_x^{\frac{2d}{d-2-p}}} \quad & \text{for } d\geq 7
\end{cases}
\end{equation*}
for frequencies $N\leq 10N_0$.  To simplify the formulas appearing below, we introduce the notation
\begin{align*}
\alpha(d):=
\begin{cases}
-\frac{d-3}2+\frac2p \quad & \text{for} \quad d=5,6\\
-\frac{d-2}2+\frac2p+\frac p2 \quad & \text{for} \quad d\geq 7.
\end{cases}
\end{align*}
Note that by \eqref{sc}, $0<\alpha(d)<1$ for $d=5,6$ and $0<\alpha(d)<p$ for $d\geq 7$.
Note also that by Bernstein's inequality and \eqref{Hs bounded},
\begin{align}\label{A bdd}
A(N)\lesssim N^{\frac{d}2-\frac2p} \|u_N\|_{L_t^\infty L_x^2}
\lesssim \bigl\| |\nabla|^{s_c} u\bigr\|_{L_t^\infty L_x^2} <\infty.
\end{align}

We next prove a recurrence formula for $A(N$).

\begin{lemma}[Recurrence]\label{L:recurrence}
For all $N\leq 10 N_0$,
\begin{align*}
A(N)
&\lesssim_u \bigl(\tfrac{N}{N_0}\bigr)^{\min\{1,p\}-\alpha(d)}
    + \eta^p \!\!\! \sum_{\frac{N}{10}\leq N_1\leq N_0} \bigl(\tfrac{N}{N_1}\bigr)^{\min\{1,p\}-\alpha(d)}A(N_1)
    +\eta^p \!\!\! \sum_{N_1<\frac{N}{10}} \bigl(\tfrac{N_1}{N}\bigr)^{\alpha(d)}A(N_1).
\end{align*}
\end{lemma}

\begin{proof}
We first give the proof in dimensions $d\geq 7$.  Once this is completed, we will explain the changes necessary to treat $d=5,6$.

Fix $N\leq 10 N_0$.  By time-translation symmetry, it suffices to prove
\begin{align}\label{rec goal}
N^{-\alpha(d)} \| u_N(0)\|_{L_x^{\frac{2d}{d-2-p}}} &\lesssim_u \bigl(\tfrac{N}{N_0}\bigr)^{p-\alpha(d)}
+ \eta^p \sum_{\frac{N}{10}\leq N_1\leq N_0} \bigl(\tfrac{N}{N_1}\bigr)^{p-\alpha(d)}A(N_1) \notag\\
&\quad + \eta^p \sum_{N_1<\frac{N}{10}} \bigl(\tfrac{N_1}{N}\bigr)^{\alpha(d)}A(N_1).
\end{align}

Using the Duhamel formula \eqref{Duhamel} into the future followed by the triangle inequality, Bernstein, and the dispersive
inequality, we estimate
\begin{align}\label{to estimate}
N^{-\alpha(d)} \| u_N(0)\|_{L_x^{\frac{2d}{d-2-p}}}
&\leq N^{-\alpha(d)} \int_0^{N^{-2}} \bigl\|e^{-it\Delta} P_N F(u(t))\bigr\|_{L_x^{\frac{2d}{d-2-p}}} \, dt \notag\\
&\quad + N^{-\alpha(d)}  \int_{N^{-2}}^{\infty}\bigl\|e^{-it\Delta}  P_N F(u(t))\bigr\|_{L_x^{\frac{2d}{d-2-p}}} \, dt \notag \\
&\lesssim N^{\frac p2 +1 -\alpha(d)} \int_0^{N^{-2}} \bigl\| e^{-it\Delta}  P_N F(u(t))\bigr\|_{L_x^2} \, dt  \notag \\
&\quad + N^{-\alpha(d)} \| P_N F(u)\|_{L_t^\infty L_x^{\frac{2d}{d+2+p}}} \int_{N^{-2}}^\infty t^{-\frac{2+p}2}\, dt  \notag \\
&\lesssim N^{\frac p2-1-\alpha(d)} \| P_N F(u)\|_{L_t^\infty L_x^2} + N^{p-\alpha(d)} \|  P_N  F(u)\|_{L_t^\infty L_x^{\frac{2d}{d+2+p}}}\notag\\
&\lesssim N^{p-\alpha(d)} \|  P_N  F(u)\|_{L_t^\infty L_x^{\frac{2d}{d+2+p}}}.
\end{align}

Using the Fundamental Theorem of Calculus, we decompose
\begin{align}\label{decomposition}
F(u)&= O(|u_{>N_0}| |u_{\leq N_0}|^p) + O(|u_{>N_0}|^{1+p}) + F(u_{\frac N{10}\leq \cdot\leq N_0})\notag \\
&\quad + u_{<\frac N{10}} \int_0^1 F_z \bigl( u_{\frac N{10}\leq \cdot\leq N_0} + \theta u_{<\frac N{10}}\bigr)\, d\theta  \\
&\quad + \overline{u_{<\frac N{10}}} \int_0^1 F_{\bar z} \bigl( u_{\frac N{10}\leq \cdot\leq N_0} + \theta u_{<\frac N{10}}\bigr)\, d\theta. \notag
\end{align}

The contribution to the right-hand side of \eqref{to estimate} coming from terms that contain at least one copy of $u_{> N_0}$
can be estimated in the following manner:  Using H\"older, Bernstein, Sobolev embedding, and \eqref{Hs bounded},
\begin{align}\label{1}
N^{p-\alpha(d)} \| P_N O(|u_{>N_0}| |u|^p) \bigr\|_{L_t^\infty L_x^{\frac{2d}{d+2+p}}}
&\lesssim N^{p-\alpha(d)}\|u\|_{L_t^\infty L_x^{\frac{dp}2}}^p \|u_{>N_0}\|_{L_t^\infty L_x^{\frac{2d}{d-2+p}}} \notag\\
&\lesssim_u N^{p-\alpha(d)} N_0^{\alpha(d)-p}.
\end{align}
Thus, this contribution is acceptable.

Next we turn to the contribution to the right-hand side of \eqref{to estimate} coming from the last two terms in
\eqref{decomposition}; it suffices to consider the first of them since similar arguments can be used to deal with the second.

First we note that as $|\nabla|^{s_c} u \in L_t^\infty L_x^2$, by Sobolev embedding we must have $\nabla u \in L_t^\infty L_x^{\frac{dp}{2+p}}$.
Thus, an application of Lemma~\ref{L:Nonlin Bernstein} together with \eqref{Hs bounded} yield
$$
\bigl\|P_{>\frac N{10}}F_z(u)\bigr\|_{L_t^\infty L_x^{\frac d{2+p}}}
\lesssim N^{-p} \|\nabla u\|_{L_t^\infty L_x^{\frac{dp}{2+p}}}^p
\lesssim N^{-p} \bigl\| |\nabla|^{s_c}u\bigr\|_{L_t^\infty L_x^2}^p
\lesssim_u N^{-p}.
$$
Thus, by H\"older's inequality and \eqref{Hs small},
\begin{align}\label{2}
N^{p-\alpha(d)}&\Bigl\| P_N\Bigl( u_{<\frac N{10}}  \int_0^1 F_z \bigl( u_{\frac N{10}\leq \cdot\leq N_0}
            + \theta u_{<\frac N{10}}\bigr)\, d\theta\Bigr) \Bigr\|_{L_t^\infty L_x^{\frac{2d}{d+2+p}}}\notag\\
&\lesssim N^{p-\alpha(d)}\|u_{<\frac N{10}}\|_{L_t^\infty L_x^{\frac{2d}{d-2-p}}}
        \Bigl\| P_{>\frac N{10}}\Bigl(\int_0^1 F_z \bigl( u_{\frac N{10}\leq \cdot\leq N_0}+ \theta u_{<\frac N{10}}\bigl)\, d\theta\Bigr)\Bigr\|_{L_t^\infty L_x^{\frac d{2+p}}}\notag\\
&\lesssim N^{-\alpha(d)} \|u_{<\frac N{10}}\|_{L_t^\infty L_x^{\frac{2d}{d-2-p}}}\bigl\| |\nabla|^{s_c}u_{\leq N_0}\bigr\|_{L_t^\infty L_x^2}^p\notag\\
&\lesssim_u \eta^p  \sum_{N_1<\frac{N}{10}} \bigl(\tfrac{N_1}{N}\bigr)^{\alpha(d)}A(N_1).
\end{align}
Hence, the contribution coming from the last two terms in \eqref{decomposition} is acceptable.

We are left to estimate the contribution of $F(u_{\frac N{10}\leq \cdot\leq N_0})$ to the right-hand side of \eqref{to
estimate}. We need only show
\begin{align}\label{last goal}
\|F(u_{\frac N{10}\leq \cdot\leq N_0})\|_{L_t^\infty L_x^{\frac{2d}{d+2+p}}}
\lesssim_u \eta^p \sum_{\frac{N}{10} \leq N_1\leq N_0} N_1^{\alpha(d)-p}A(N_1).
\end{align}
As $d\geq 7$ and the critical regularity $s_c$ satisfies \eqref{sc}, we must have $p<1$.  Using the triangle inequality, Bernstein, \eqref{Hs small}, and H\"older, we
estimate
\begin{align*}
\| F(u_{\frac N{10}\leq \cdot\leq N_0}& )\|_{L_t^\infty L_x^{\frac{2d}{d+2+p}}}\\
&\lesssim \sum_{\frac N{10}\leq N_1 \leq N_0} \bigl\|u_{N_1} |u_{\frac N{10}\leq \cdot\leq N_0}|^p\bigr\|_{L_t^\infty L_x^{\frac{2d}{d+2+p}}}\\
&\lesssim \sum_{\frac N{10}\leq N_1, N_2 \leq N_0} \bigl\|u_{N_1} |u_{N_2}|^p\bigr\|_{L_t^\infty L_x^{\frac{2d}{d+2+p}}}\\
&\lesssim \sum_{\frac N{10}\leq N_1\leq N_2 \leq N_0} \|u_{N_1}\|_{L_t^\infty L_x^{\frac{2d}{d-2-p}}} \|u_{N_2}\|_{L_t^\infty L_x^{\frac{dp}{2+p}}}^p\\
&\quad + \sum_{\frac N{10}\leq N_2\leq N_1 \leq N_0} \|u_{N_1}\|_{L_t^\infty L_x^{\frac{dp}{2+p}}}^p
           \|u_{N_1}\|_{L_t^\infty L_x^{\frac{2d}{d-2-p}}}^{1-p} \|u_{N_2}\|_{L_t^\infty L_x^{\frac{2d}{d-2-p}}}^{p}\\
&\lesssim \sum_{\frac N{10}\leq N_1\leq N_2 \leq N_0} \|u_{N_1}\|_{L_t^\infty L_x^{\frac{2d}{d-2-p}}} N_2^{-p}\bigl\| |\nabla|^{s_c}u_{N_2}\bigr\|_{L_t^\infty L_x^2}^p\\
&\quad + \sum_{\frac N{10}\leq N_2\leq N_1 \leq N_0}N_1^{-p}\bigl\| |\nabla|^{s_c}u_{N_1}\bigr\|_{L_t^\infty L_x^2}^p
            \|u_{N_1}\|_{L_t^\infty L_x^{\frac{2d}{d-2-p}}}^{1-p} \|u_{N_2}\|_{L_t^\infty L_x^{\frac{2d}{d-2-p}}}^{p}\\
&\lesssim_u \eta^p \sum_{\frac N{10}\leq N_1\leq N_0} N_1^{\alpha(d)-p}A(N_1)\\
&\quad + \eta^p \sum_{\frac N{10}\leq N_2\leq N_1 \leq N_0} \bigl(\tfrac{N_2}{N_1} \bigr)^{p^2}
            \bigl[N_1^{\alpha(d)-p}A(N_1)\bigr]^{1-p} \bigl[N_2^{\alpha(d)-p}A(N_2)\bigr]^p\\
&\lesssim_u \eta^p \sum_{\frac N{10}\leq N_1\leq N_0} N_1^{\alpha(d)-p}A(N_1).
\end{align*}
This proves \eqref{last goal} and so completes the proof of the lemma in dimensions $d\geq 7$.

Consider now $d=5,6$.  Note that in this case, our assumptions guarantee that $1<p<2$. Arguing as for \eqref{to estimate}, we have
$$
N^{-\alpha(d)} \|u_N(0)\|_{L_x^{\frac{2d}{d-3}}}\lesssim N^{1-\alpha(d)} \|P_N F(u)\|_{L_t^\infty L_x^{\frac{2d}{d+3}}},
$$
which we estimate by decomposing the nonlinearity as in \eqref{decomposition}.  The analogue of \eqref{1} in this case is
\begin{align*}
N^{1-\alpha(d)} \| P_N O(|u_{>N_0}| |u|^p) \bigr\|_{L_t^\infty L_x^{\frac{2d}{d+3}}}
&\lesssim N^{1-\alpha(d)} \|u\|_{L_t^\infty L_x^{\frac{dp}2}}^p \|u_{>N_0}\|_{L_t^\infty L_x^{\frac{2d}{d-1}}}
\lesssim_u \bigl(\tfrac{N}{N_0}\bigr)^{1-\alpha(d)}.
\end{align*}
Using Bernstein and Lemma~\ref{L:Fract chain rule} together with \eqref{Hs small} and Sobolev embedding, we replace \eqref{2} by
\begin{align*}
N^{1-\alpha(d)} \Bigl\| & P_N\Bigl( u_{<\frac N{10}}  \int_0^1 F_z \bigl( u_{\frac N{10}\leq \cdot\leq N_0}
        + \theta u_{<\frac N{10}}\bigr)\, d\theta\Bigr) \Bigr\|_{L_t^\infty L_x^{\frac{2d}{d+3}}}\\
&\lesssim N^{1-\alpha(d)}\|u_{<\frac N{10}}\|_{L_t^\infty L_x^{\frac{2d}{d-3}}}
        \Bigl\| P_{>\frac N{10}}\Bigl(\int_0^1 F_z \bigl( u_{\frac N{10}\leq \cdot\leq N_0}+ \theta u_{<\frac N{10}}\bigl)\, d\theta\Bigr)\Bigr\|_{L_t^\infty L_x^{\frac d3}}\\
&\lesssim N^{-\alpha(d)} \|u_{<\frac N{10}}\|_{L_t^\infty L_x^{\frac{2d}{d-3}}} \|\nabla u_{\leq N_0}\|_{L_t^\infty L_x^{\frac{dp}{2+p}}}
        \|u_{\leq N_0}\|_{L_t^\infty L_x^{\frac{dp}2}}^{p-1} \\
&\lesssim \sum_{N_1<\frac{N}{10}} \bigl(\tfrac{N_1}{N}\bigr)^{\alpha(d)}A(N_1)\bigl\| |\nabla|^{s_c}u_{\leq N_0}\bigr\|_{L_t^\infty L_x^2}^p\\
&\lesssim_u \eta^p \sum_{N_1<\frac{N}{10}} \bigl(\tfrac{N_1}{N}\bigr)^{\alpha(d)}A(N_1).
\end{align*}
Finally, arguing as for \eqref{last goal}, we estimate
\begin{align*}
\|F&( u_{\frac N{10}\leq \cdot\leq N_0})\|_{L_t^\infty L_x^{\frac{2d}{d+3}}}\\
&\lesssim \sum_{\frac N{10}\leq N_1, N_2 \leq N_0}
    \bigl\|u_{N_1} u_{N_2}|u_{\frac N{10}\leq \cdot\leq N_0}|^{p-1}\bigr\|_{L_t^\infty L_x^{\frac{2d}{d+3}}}\\
&\lesssim \sum_{\frac N{10}\leq N_1\leq N_2, N_3 \leq N_0} \|u_{N_1}\|_{L_t^\infty L_x^{\frac{2d}{d-3}}}
    \|u_{N_2}\|_{L_t^\infty L_x^{\frac{dp}3}}\|u_{N_3}\|_{L_t^\infty L_x^{\frac{dp}3}}^{p-1}\\
&\qquad + \sum_{\frac N{10}\leq N_3\leq N_1\leq N_2 \leq N_0} \|u_{N_1}\|_{L_t^\infty L_x^{\frac{2d}{d-3}}}^{2-p}
    \|u_{N_1}\|_{L_t^\infty L_x^{\frac{dp}3}}^{p-1}\|u_{N_2}\|_{L_t^\infty L_x^{\frac{dp}3}}\|u_{N_3}\|_{L_t^\infty L_x^{\frac{2d}{d-3}}}^{p-1}\\
&\lesssim_u \sum_{\frac N{10}\leq N_1\leq N_2, N_3 \leq N_0} \|u_{N_1}\|_{L_t^\infty L_x^{\frac{2d}{d-3}}}
    \eta N_2^{-\frac1p}\bigl(\eta N_3^{-\frac1p}\bigr)^{p-1}\\
&\qquad + \sum_{\frac N{10}\leq N_3\leq N_1\leq N_2 \leq N_0} \|u_{N_1}\|_{L_t^\infty L_x^{\frac{2d}{d-3}}}^{2-p} \bigl(\eta N_1^{-\frac 1p}\bigr)^{p-1}
    \eta N_2^{-\frac1p} \|u_{N_3}\|_{L_t^\infty L_x^{\frac{2d}{d-3}}}^{p-1}\\
&\lesssim_u \eta^p \sum_{\frac N{10}\leq N_1 \leq N_0} N_1^{\alpha(d)-1} A(N_1)\\
&\qquad + \eta^p \sum_{\frac N{10}\leq N_3\leq N_1 \leq N_0} \bigl( \tfrac{N_3}{N_1}\bigr)^{p-1}
            \bigl[N_1^{\alpha(d)-1}A(N_1)\bigr]^{2-p} \bigl[N_3^{\alpha(d)-1}A(N_3)\bigr]^{p-1}\\
&\lesssim_u \eta^p \sum_{\frac N{10}\leq N_1 \leq N_0} N_1^{\alpha(d)-1} A(N_1).
\end{align*}
Putting everything together completes the proof of the lemma when $d=5,6$.
\end{proof}

This lemma leads very quickly to our first goal:

\begin{proposition}[$L^p$ breach of scaling]\label{P:L^p breach}
Let $u$ be as in Theorem~\ref{T:-reg}.  Then
\begin{align}\label{step 1}
u\in L_t^\infty L_x^q \quad \text{for some} \quad 2<q<\tfrac{dp}2.
\end{align}
In particular, for $0\leq s\leq s_c$,
\begin{align}\label{breach 2}
|\nabla|^s u\in L_t^\infty L_x^2 \quad \implies \quad |\nabla|^{s} F(u) \in L_t^\infty L_x^r \quad \text{for some} \quad r<\tfrac{2d}{d+4}.
\end{align}
\end{proposition}

\begin{proof}
We first consider \eqref{step 1}.  We will only present the details for $d\geq 7$.  The treatment of $d=5,6$ is completely analogous.

Combining Lemma~\ref{L:recurrence} with Lemma~\ref{L:Gronwall}, we deduce
\begin{align}\label{breach 1}
\|u_N\|_{L_t^\infty L_x^{\frac{2d}{d-2-p}}}\lesssim_u N^{p-} \quad \text{for all} \quad N\leq 10 N_0.
\end{align}
In applying Lemma~\ref{L:Gronwall}, we set $N=10 \cdot 2^{-k} N_0$, $x_k=A(10 \cdot 2^{-k} N_0)$, and take $\eta$ sufficiently
small.  Note that $x_k\in \ell^\infty$ by virtue of \eqref{A bdd}.

By interpolation followed by \eqref{breach 1}, Bernstein, and \eqref{Hs bounded},
\begin{align*}
\|u_N\|_{L_t^\infty L_x^q}
&\lesssim \|u_N\|_{L_t^\infty L_x^{\frac{2d}{d-2-p}}}^{\frac{d(q-2)}{q(2+p)}} \|u_N\|_{L_t^\infty L_x^2}^{\frac{2d-q(d-2-p)}{q(2+p)}}
\lesssim_u N^{\frac{dp(q-2)}{q(2+p)}-} N^{-(\frac d2-\frac2p)\frac{2d-q(d-2-p)}{q(2+p)}}
\end{align*}
for all $N\leq 10 N_0$.  Note that the power of $N$ appearing in the formula above is positive for
$$
\tfrac{2d(2p^2+ dp-4)}{dp^2 + (d^2-2d+4)p-4(d-2)}<q<\tfrac{dp}2.
$$
Thus, for these values of $q$, Bernstein together with \eqref{Hs bounded} yield
\begin{align*}
\|u\|_{L_t^\infty L_x^q} \leq \|u_{\leq N_0}\|_{L_t^\infty L_x^q} + \|u_{> N_0}\|_{L_t^\infty L_x^q}
\lesssim_u \sum_{N\leq N_0} N^{0+} + \sum_{N>N_0} N^{\frac2p-\frac dq} \lesssim_{u} 1,
\end{align*}
which completes the proof of \eqref{step 1}.

We turn now to \eqref{breach 2}.  For $0\leq s\leq 1$, we use Lemma~\ref{L:Fract chain rule} to estimate
\begin{align*}
\bigl\| |\nabla|^s  F(u)\bigr\|_{L_t^\infty L_x^r}
\lesssim \bigl\| |\nabla|^s u\bigr\|_{L_t^\infty L_x^2} \|u\|_{L_t^\infty L_x^{\frac{2rp}{2-r}}}^p.
\end{align*}
We will return to this shortly, but first we consider the case $1<s\leq s_c$.  Using Lemma~\ref{L:product rule} together with Sobolev embedding,
\begin{align*}
\bigl\| |\nabla|^s  F(u)\bigr\|_{L_t^\infty L_x^r}
&\lesssim \bigl\| |\nabla|^{s-1} \bigl(\nabla u \cdot F'(u)\bigr)\bigr\|_{L_t^\infty L_x^r}\\
&\lesssim \bigl\| |\nabla|^{s} u\bigr\|_{L_t^\infty L_x^2} \|u\|_{L_t^\infty L_x^{\frac{2rp}{2-r}}}^p
    +\|\nabla u\|_{L_t^\infty L_x^{\frac{dp}{2+p}}}
        \bigl\| |\nabla|^{s-1} F'(u)\bigr\|_{L_t^\infty L_x^{\frac{dpr}{dp-r(2+p)}}}\\
&\lesssim \bigl\| |\nabla|^{s} u\bigr\|_{L_t^\infty L_x^2} \Bigl(\|u\|_{L_t^\infty L_x^{\frac{2rp}{2-r}}}^p
        +\bigl\| |\nabla|^{s-1} F'(u)\bigr\|_{L_t^\infty L_x^{\frac{dpr}{dp-r(2+p)}}}\Bigr).
\end{align*}
For $p\geq 1$, we invoke Lemma~\ref{L:Fract chain rule} and then use Sobolev embedding to estimate
\begin{align*}
\bigl\| |\nabla|^{s-1} F'(u)\bigr\|_{L_t^\infty L_x^{\frac{dpr}{dp-r(2+p)}}}
&\lesssim \bigl\| |\nabla|^{s-1}u \bigr\|_{L_t^\infty L_x^{\frac{2d}{d-2}}} \|u\|_{L_t^\infty L_x^{\frac{2dpr(p-1)}{2dp-dpr-4r}}}^{p-1}\\
&\lesssim \bigl\| |\nabla|^{s}u \bigr\|_{L_t^\infty L_x^2} \|u\|_{L_t^\infty L_x^{\frac{2dpr(p-1)}{2dp-dpr-4r}}}^{p-1}.
\end{align*}
If instead $p<1$, we use Lemma~\ref{L:FDFP} and Sobolev embedding to obtain
\begin{align*}
\bigl\||\nabla|^{s-1} F'(u)\bigr\|_{L_t^\infty L_x^{\frac{dpr}{dp-r(2+p)}}}
&\lesssim \bigl\| |\nabla|^\sigma u \bigr\|_{L_t^\infty L_x^{\frac{dp}{2+\sigma p}}}^{\frac{s-1}\sigma}
    \|u\|_{L_t^\infty L_x^{\frac{dpr(\sigma p + 1-s)}{dp\sigma-\sigma prs-2r(s-1)-2\sigma r}}}^{p-\frac{s-1}\sigma}\\
&\lesssim \bigl\| |\nabla|^{s} u \bigr\|_{L_t^\infty L_x^2}^{\frac{s-1}\sigma}
    \|u\|_{L_t^\infty L_x^{\frac{dpr(\sigma p + 1-s)}{dp\sigma-\sigma prs-2r(s-1)-2\sigma r}}}^{p-\frac{s-1}\sigma}
\end{align*}
for some $\frac{s-1}p<\sigma<1$.

For all $0\leq s\leq s_c$, to obtain \eqref{breach 2} it thus suffices to note that for $r<\frac{2d}{d+4}$ the exponents
$\frac{2rp}{2-r}$, $\frac{2dpr(p-1)}{2dp-dpr-4r}$, $\frac{dpr(\sigma p + 1-s)}{dp\sigma-\sigma prs-2r(s-1)-2\sigma r}$
are all less than $\frac{dp}2$ and converge to this number as $r\to \frac{2d}{d+4}$.  Thus, \eqref{step 1} yields the claim.
\end{proof}

Following \cite{Berbec}, our second step is to use the double Duhamel trick to upgrade \eqref{step 1} to a negative regularity statement
in $L_x^2$-based Sobolev spaces.  The arguments we present follow closely the ones in \cite{Berbec}.  For the sake of completeness, we
present the details below.

\begin{proposition}[Some negative regularity]\label{P:some -reg}
Let $d\geq 5$ and assume the critical regularity $s_c$ obeys \eqref{sc}.  Let $u$ be as in Theorem~\ref{T:-reg}.
Assume further that $|\nabla|^s F(u) \in L_t^\infty L_x^r$ for some $r<\tfrac{2d}{d+4}$ and some $0\leq s\leq s_c$.
Then there exists $s_0=s_0(d,r)>0$ such that $u \in L_t^\infty \dot H^{s-s_0+}_x$.
\end{proposition}

\begin{proof}
The proposition will follow once we establish
\begin{align}\label{neg reg 0}
\bigl\| |\nabla|^s u_N \bigr\|_{L_t^\infty L_x^2} \lesssim_u N^{s_0} \quad \text{for all} \quad N>0 \quad \text{and} \quad
s_0:=\tfrac{d}r-\tfrac{d+4}2>0.
\end{align}
Indeed, by Bernstein combined with \eqref{Hs bounded},
\begin{align*}
\bigl\| |\nabla|^{s-s_0+} u \bigr\|_{L_t^\infty L_x^2}
&\leq \bigl\| |\nabla|^{s-s_0+} u_{\leq 1} \bigr\|_{L_t^\infty L_x^2} + \bigl\| |\nabla|^{s-s_0+} u_{>1} \bigr\|_{L_t^\infty L_x^2}\\
&\lesssim_u \sum_{N\leq 1} N^{0+} + \sum_{N>1} N^{(s-s_0+)-s_c}\\
&\lesssim_u 1.
\end{align*}

Thus, we are left to prove \eqref{neg reg 0}.  By time-translation symmetry, it suffices to prove
\begin{align}\label{neg reg 1}
\bigl\| |\nabla|^s u_N(0) \bigr\|_{L_x^2} \lesssim_u N^{s_0} \quad \text{for all} \quad N>0 \quad \text{and} \quad
s_0:=\tfrac{d}r-\tfrac{d+4}2>0.
\end{align}
Using the Duhamel formula \eqref{Duhamel} both in the future and in the past, we write
\begin{align*}
\bigl\| |\nabla|^s & u_N(0) \bigr\|_{L_x^2}^2 \\
&= \lim_{T\to\infty} \lim_{T'\to-\infty} \bigl \langle i\int_0^T e^{-it\Delta}P_N |\nabla|^s F(u(t))\, dt,
    -i\int_{T'}^0 e^{-i\tau\Delta}P_N |\nabla|^s F(u(\tau))\, d\tau  \bigr\rangle\\
&\leq  \int_0^{\infty} \int_{-\infty}^0 \Bigl| \bigl\langle P_N |\nabla|^s F(u(t)) ,
    e^{i(t-\tau)\Delta}P_N |\nabla|^s F(u(\tau))  \bigr\rangle \Bigr| \,dt\, d\tau.
\end{align*}
We estimate the term inside the integrals in two ways.  On one hand, using H\"older and the dispersive estimate,
\begin{align*}
\Bigl|\bigl\langle P_N |\nabla|^s F(u(t)) ,  e^{i(t-\tau)\Delta} & P_N |\nabla|^s F(u(\tau))  \bigr\rangle\Bigr|\\
&\lesssim \bigl\| P_N |\nabla|^s F(u(t))\bigr\|_{L_x^r} \bigl\| e^{i(t-\tau)\Delta} P_N |\nabla|^s F(u(\tau))\bigr\|_{L_x^{r'}}\\
&\lesssim |t-\tau|^{\frac d2-\frac dr} \bigl\| |\nabla|^s F(u)\bigr\|_{L_t^\infty L_x^r}^2.
\end{align*}
On the other hand, using Bernstein,
\begin{align*}
\Bigl|\bigl\langle P_N |\nabla|^s F(u(t)) ,  e^{i(t-\tau)\Delta} & P_N |\nabla|^s F(u(\tau))  \bigr\rangle\Bigr|\\
&\lesssim \bigl\| P_N |\nabla|^s F(u(t))\bigr\|_{L_x^2} \bigl\| e^{i(t-\tau)\Delta} P_N |\nabla|^s F(u(\tau))\bigr\|_{L_x^2}\\
&\lesssim N^{2(\frac d2-\frac dr)} \bigl\| |\nabla|^s F(u)\bigr\|_{L_t^\infty L_x^r}^2.
\end{align*}
Thus,
\begin{align*}
\bigl\| |\nabla|^s u_N(0) \bigr\|_{L_x^2}^2
&\lesssim \bigl\| |\nabla|^s F(u)\bigr\|_{L_t^\infty L_x^r}^2 \int_0^\infty \int_{-\infty}^0 \min\{ |t-\tau|^{-1} , N^2 \}^{\frac dr -\frac d2} \, dt\, d\tau\\
&\lesssim N^{2s_0} \bigl\| |\nabla|^s F(u)\bigr\|_{L_t^\infty L_x^r}^2.
\end{align*}
To obtain the last inequality we used the fact that $\tfrac dr -\tfrac d2 >2$ since $r<\tfrac{2d}{d+4}$. Thus \eqref{neg reg 1}
holds; this finishes the proof of the proposition.
\end{proof}

The proof of Theorem~\ref{T:-reg} will follow from iterating Proposition~\ref{P:some -reg} finitely many times.

\begin{proof}[Proof of Theorem~\ref{T:-reg}]
Proposition~\ref{P:L^p breach} allows us to apply Proposition~\ref{P:some -reg} with $s=s_c$. We conclude that $u \in L_t^\infty
\dot H^{s_c-s_0+}_x$ for some $s_0=s_0(d,r)>0$.  Combining this with \eqref{breach 2}, we deduce that
$|\nabla|^{s_c-s_0+} F(u) \in L_t^\infty L_x^r$ for some $r<\tfrac{2d}{d+4}$.  We are thus in the position to apply Proposition~\ref{P:some -reg}
again and obtain $u \in L_t^\infty \dot H^{s_c-2s_0+}_x$.  Iterating this procedure finitely many times, we derive
$u\in L_t^\infty \dot H_x^{-\eps}$ for any $0<\eps<s_0$.

This completes the proof of Theorem~\ref{T:-reg}.
\end{proof}

To prove the equivalent of Theorem~\ref{T:-reg} in the setting of Theorem~\ref{T:main cubic}, one repeats the arguments presented above with
\begin{equation*}
A(N):= N^{\frac{d-5}2} \sup_{t\in \R} \|u_N(t)\|_{L_x^{\frac{2d}{d-3}}}.
\end{equation*}

%
%
%
%

\section{The low-to-high frequency cascade}\label{S:cascade}

In this section, we use the negative regularity proved in the previous section to preclude low-to-high frequency cascade
solutions.

\begin{theorem}[Absence of cascades]\label{T:no cascade}
In the settings of Theorems~\ref{T:main cubic} and \ref{T:main} there are no solutions to \eqref{nls}
that are low-to-high frequency cascades in the sense of Theorem~\ref{T:enemies}.
\end{theorem}

\begin{proof}
Suppose for a contradiction that there existed such a solution $u$.  Then, by the results proved in Section~\ref{S:neg}, $u\in L_t^\infty L_x^2$;
thus, by the conservation of mass,
$$
0\leq M(u)=M(u(t)) = \int_{\R^d} |u(t,x)|^2\, dx <\infty \quad \text{for all} \quad t\in \R.
$$

Fix $t\in \R$ and let $\eta>0$ be a small constant.  By compactness (see Remark~\ref{R:c small}),
$$
\int_{|\xi|\leq c(\eta)N(t)} |\xi|^{2s_c} |\hat u(t,\xi)|^2\, d\xi\leq \eta.
$$
On the other hand, as $u\in L_t^\infty \dot H^{-\eps}_x$ for some $\eps>0$,
$$
\int_{|\xi|\leq c(\eta)N(t)} |\xi|^{-2\eps} |\hat u(t,\xi)|^2\, d\xi\lesssim_u 1.
$$
Hence, by H\"older's inequality,
\begin{align}\label{close}
\int_{|\xi|\leq c(\eta)N(t)} |\hat u(t,\xi)|^2\, d\xi\lesssim_u \eta^{\frac \eps{s_c+\eps}}.
\end{align}

Meanwhile, by elementary considerations and $u\in L_t^\infty \dot H^{s_c}_x$,
\begin{align}\label{far}
\int_{|\xi|\geq c(\eta)N(t)} |\hat u(t,\xi)|^2\, d\xi
&\leq [c(\eta)N(t)]^{-2s_c} \int_{\R^d} |\xi|^{2s_c} |\hat u(t,\xi)|^2\, d\xi \notag \\
&\leq [c(\eta)N(t)]^{-2s_c} \bigl\||\nabla|^{s_c} u(t)\bigr\|^2_2 \notag\\
&\lesssim_u [c(\eta)N(t)]^{-2s_c}.
\end{align}

Collecting \eqref{close} and \eqref{far} and using Plancherel's theorem, we obtain
$$
0\leq M(u)\lesssim_u c(\eta)^{-2s_c} N(t)^{-2s_c} + \eta^{\frac \eps{s_c+\eps}}
$$
for all $t\in \R$.  As $u$ is a low-to-high cascade, there is a sequence of times $t_n\to\infty$ so that $N(t_n)\to \infty$.  As
$\eta>0$ is arbitrary, we may conclude $M(u)=0$ and hence $u$ is identically zero. This contradicts the fact that
$S_\R(u)=\infty$, thus settling Theorem~\ref{T:no cascade}.
\end{proof}

%
%
%
%

\section{The soliton}\label{S:Soliton}

In this section, we use the negative regularity proved in Section~\ref{S:neg} to preclude soliton-like solutions.

\begin{theorem}[Absence of solitons]\label{T:no soliton}
In the settings of Theorems~\ref{T:main cubic} and \ref{T:main} there are no global solutions to \eqref{nls}
that are solitons in the sense of Theorem~\ref{T:enemies}.
\end{theorem}

\begin{proof}
As usual, we will use a monotonicity formula to preclude soliton-like solutions.  Since we are in the defocusing case, we
will use the interaction Morawetz inequality introduced in \cite{CKSTT:interact}.  For the high dimensional case discussed here,
the details of this derivation can be found in \cite{matador} or \cite{thesis:art}.

To prove Theorem~\ref{T:no soliton}, we argue by contradiction.  We assume there exists a solution $u$ to \eqref{nls} which is a soliton
in the sense of Theorem~\ref{T:enemies}.  Then, by the negative regularity results proved in Section~\ref{S:neg}, $u\in L_t^\infty H_x^1$.
The interaction Morawetz inequality yields
\begin{align*}
\int_I \int_{\R^d}\int_{\R^d} \frac{|u(t,x)|^2 |u(t,y)|^2}{|x-y|^3}\, dx\, dy\, dt
\lesssim \|u\|_{L_t^\infty L_x^2(I\times\R^d)}^3 \|\nabla u\|_{L_t^\infty L_x^2(I\times\R^d)}
\lesssim_u 1
\end{align*}
for any compact time interval $I\subset \R$.  As in dimension $d$ convolution with $|x|^{-3}$ is basically the same as the
fractional integration operator $|\nabla|^{-(d-3)}$, the interaction Morawetz inequality yields
\begin{align*}
\bigl\| |\nabla|^{-\frac{d-3}2} (|u|^2) \bigr\|_{L^2_{t,x}(I \times \R^d)} \lesssim_u 1.
\end{align*}
By \cite[Lemma 5.6]{thesis:art}, this implies
\begin{align*}
\bigl\| |\nabla|^{-\frac{d-3}4} u \bigr\|_{L_{t,x}^4(I \times \R^d)} \lesssim_u 1.
\end{align*}
Interpolating between this estimate and the fact that $u\in L_t^\infty \dot H^1_x$, we derive
\begin{align}\label{Morawetz bdd}
\|u\|_{L_t^{d+1} L_x^{\frac{2(d+1)}{d-1}}(I \times \R^d)} \lesssim_u 1
\end{align}
for any compact time interval $I\subset \R$.

Next we claim that
\begin{align}\label{concentration}
\|u(t)\|_{L_x^{\frac{2(d+1)}{d-1}}}\gtrsim_u 1 \quad \text{uniformly for } t\in\R.
\end{align}
Otherwise, there exists a time sequence $t_n$ such that $u(t_n)$ converges weakly to zero in $L_x^{\frac{2(d+1)}{d-1}}$.
As $u(t)$ is uniformly bounded in $\dot H^{s_c}_x$, this implies that $u(t_n)$ converges weakly to zero in $\dot H^{s_c}_x$.
As the orbit of $u$ is precompact in $\dot H^{s_c}_x$ and $u$ is not identically zero, we derive a contradiction.

Using \eqref{Morawetz bdd} and \eqref{concentration}, we easily derive a contradiction by taking the interval $I$ to be sufficiently long.
\end{proof}

\end{document}